\documentclass[reqno,12pt]{amsart}

\usepackage{amsmath,mathtools,slashed,bbm,appendix,enumitem,verbatim}
\usepackage[protrusion=true,babel=true]{microtype}
\usepackage[english]{babel}
\usepackage[widespace,upright]{fourier}
\usepackage[backrefs]{amsrefs}
\usepackage[margin=1in]{geometry}
\usepackage[onehalfspacing]{setspace}
\usepackage[pdfusetitle,pagebackref]{hyperref}
\numberwithin{equation}{section}				% must call it before cleveref
\usepackage[nameinlink,noabbrev]{cleveref}
\expandafter\def\csname ver@etex.sty\endcsname{3000/12/31}

\usepackage{autonum}						% must call after cleveref

\allowdisplaybreaks[1]

\let\originalleft\left
\let\originalright\right
\renewcommand{\left}{\mathopen{}\mathclose\bgroup\originalleft}
\renewcommand{\right}{\aftergroup\egroup\originalright}

\def\({\mathopen{}\left(}
\def\){\right)\mathclose{}}

\makeatletter
\renewcommand*{\eqref}[1]{\hyperref[{#1}]{\textup{\tagform@{\ref*{#1}}}}}
\makeatother

\newcommand*{\eqdef}{\mathrel{\vcenter{\baselineskip0.5ex \lineskiplimit0pt\hbox{.}\hbox{.}}}=}

\newtheorem*{acknowledgment}{Acknowledgment}

\newtheorem{theorem}{Theorem}[section]

\newtheorem{lemma}[theorem]{Lemma}
\newtheorem{corollary}[theorem]{Corollary}
\newtheorem{remark}[theorem]{Remark}
\newtheorem{definition}[theorem]{Definition}

\newtheorem{example}[theorem]{Example}
\crefname{theorem}{Theorem}{Theorems}						% label for Theorems
\creflabelformat{theorem}{#2{#1}#3}							% label format for theorem
\crefname{main}{Main Theorem}{Main Theorems}				% label for the Main Theorems
\creflabelformat{main}{#2{#1}#3}							% label format for main
\crefname{lemma}{Lemma}{Lemmas}								% label for Lemmas
\creflabelformat{lemma}{#2{#1}#3}							% label format for lem
\crefname{corollary}{Corollary}{Corollaries}				% label for Corollaries
\creflabelformat{corollary}{#2{#1}#3}						% label format for cor
\crefname{ineq}{inequality}{inequalities}					% label for inequalities
\creflabelformat{ineq}{#2{\upshape(#1)}#3}					% label format for ineq
\crefname{cond}{condition}{conditions}						% label for conditions
\creflabelformat{cond}{#2{#1}#3}							% label format for cond
\crefname{table}{Table}{Tables}								% label for Tables
\creflabelformat{table}{#2{\upshape(#1)}#3}					% label format for Tables
\crefname{hypothesis}{Hypothesis}{Hypotheses}				% label for Hypotheses
\creflabelformat{hypothesis}{#2{#1}#3}						% label format for Hypotheses
\crefname{remark}{Remark}{Remarks}							% label for Remarks
\creflabelformat{remark}{#2{#1}#3}							% label format for Remarks
\crefname{definition}{Definition}{Definitions}				% label for Definitions
\creflabelformat{def}{#2{#1}#3}								% label format for 'def'

\def\id{\mathbbm{1}}
\def\cx{\mathbbm{C}}
\def\rl{\mathbbm{R}}
\def\N{\mathbbm{N}}
\def\P{\mathbbm{P}}
\def\Z{\mathbbm{Z}}

\def\cA{\mathcal{A}}
\def\cB{\mathcal{B}}

\def\cD{\mathcal{D}}
\def\cE{\mathcal{E}}
\def\cF{\mathcal{F}}
\def\cG{\mathcal{G}}
\def\cH{\mathcal{H}}

\def\cK{\mathcal{K}}
\def\cL{\mathcal{L}}
\def\cM{\mathcal{M}}
\def\cN{\mathcal{N}}
\def\cO{\mathcal{O}}

\def\cV{\mathcal{V}}
\def\cZ{\mathcal{Z}}

\def\Ar{\mathrm{Area}}

\def\CL{\mathrm{CL}}
\def\coker{\mathrm{coker}}
\def\rd{\mathrm{d}}
\def\dA{\: \rd A}

\def\dist{\mathrm{dist}}

\def\dom{\mathrm{dom}}
\def\End{\mathrm{End}}

\def\Hom{\mathrm{Hom}}

\def\index{\mathrm{index}}

\def\Re{\mathrm{Re}}

\def\sign{\mathrm{sign}}

\def\SO{\mathrm{SO}}

\def\Spec{\mathrm{Spec}}

\def\Sym{\mathrm{Sym}}

\def\supp{\mathrm{supp}}

\def\rU{\mathrm{U}}

\def\vol{\: \mathrm{vol}}

\def\cl{\mathfrak{cl}}
\def\del{\partial}
\def\delbar{\overline{\partial}}
\def\D{\slashed{D}}						% Dirac operator
\def\S{\slashed{S}}						% spinor bundle
\def\E{\slashed{E}}
\def\p{\mathfrak{p}}
\def\rd{\operatorname{d\!}{}}
\def\s{\mathfrak{s}}

\def\su{\mathfrak{su}}

\title{Conjugate linear perturbations of Dirac operators and Majorana fermions}
\date{\today}

\author{\'Akos Nagy}
\address[\'Akos Nagy]{University of California, Santa Barbara}
\urladdr{\href{https://akosnagy.com}{akosnagy.com}}
\email{\href{mailto:contact@akosnagy.com}{contact@akosnagy.com}}

\calclayout
\hypersetup{
	unicode			= true,
	pdffitwindow	= true,
	pdftoolbar		= false,
	pdfmenubar		= false,
	pdfstartview	= {FitH},
	hypertexnames	= true,
	colorlinks		= true,
	linkcolor		= black,
	citecolor		= black,
	filecolor		= black,
	urlcolor		= blue
}

\begin{document}

\begin{abstract}
	We study a canonical class of perturbations of Dirac operators that are defined in any dimension and on any Hermitian Clifford module bundle. These operators generalize the 2-dimensional Jackiw--Rossi operator, which describes electronic excitations on topological superconductors. We also describe the low energy spectrum of these operators on complete surfaces, under mild hypotheses.
\end{abstract}

\maketitle

\tableofcontents

\section*{Introduction}

In \cite{jackiw_zero_1981}, Jackiw and Rossi introduced a Dirac-type equation in two spatial dimensions which corresponds to a Lagrangian that couples Dirac fermions to the superconducting order parameter of an $s$-wave superconductor. Ground states of this theory are interpreted as Majorana fermions pinned to vortices \cite{CJNPS10}. Furthermore, this theory has potential applications in quantum computing; cf. \cites{ivanov_non-abelian_2001,FK08,Manna8775}.

\smallskip

In this paper, we reformulate the classical Jackiw--Rossi theory in terms of spin geometry and generalize the Jackiw--Rossi (Hamiltonian) operator to more general fields and higher dimensions, and study the spectral properties of this theory. These \emph{generalized Jackiw--Rossi operators} have the form
\begin{equation}
	H = \D + \cA,
\end{equation}
where $\D$ is a Dirac-type operator on and $\cA$ is a conjugate linear bundle map. Since $H$ is not complex linear, its eigenspaces are not complex (but only real) subspaces of the Hilbert space. Moreover, eigenspinors of $H$ can be viewed as Majorana fermions; cf. \cite{CJNPS10}. Furthermore, in certain cases the above operators have been studied in the context of pseudo-holomorphic curves; cf. \cites{taubes_rm_1996,taubes_counting_2000,R04,lee_spin_2013,gerig_generic_2017,doan_castelnuovo_2021,doan_equivariant_2023}.

\smallskip

After introducing the general theory, we study the Jackiw--Rossi equation
\begin{equation}
	H \Psi = m \Psi,
\end{equation}
on complete surfaces, with $|m|$ small. The key analytic observation in studying the Jackiw--Rossi equation is that the planar Jackiw--Rossi operator comes from a \emph{concentrating pair}, in the sense of Maridakis; cf. \cite{maridakis_spinor_2015}*{Definition~2.1}. In other words, the operator
\begin{equation}
	\cB \eqdef \D^* \circ \cA + \cA^* \circ \D,
\end{equation}
is a bundle map, as opposed to a first order differential operator. This allows us to prove strong contraction results for eigenspinors; cf. \cites{prokhorenkov_perturbations_2006,maridakis_spinor_2015}.

\smallskip

In the final section of the paper, we construct a canonical projective space bundle over symmetric powers of Riemann surfaces and we provide solutions to the (generalized) Jackiw--Rossi equation in higher dimensions.

\medskip

\begin{acknowledgment}
	I am thankful for the help of Tom Parker, who made me interested in the project during my grad school years.

	I have learned a lot about the physics related to this article from D\'aniel Varjas, and about the math from Manos Maridakis.

	I am also grateful for the organizers of the ``\emph{Mathematics of topological phases of matter}'' program at the Simons Center for Geometry and Physics, especially Dan Freed, for hosting me while I was working on parts of this project.

	I thank Xianzhe Dai for pointing out a mistake in an earlier version of the paper.

	Finally, I thank the Referees for their thorough feedback.
\end{acknowledgment}

\medskip

\subsection*{Organization of the paper:} In \Cref{sec:JR_on_manifolds}, we introduce the generalized Jackiw--Rossi Hamiltonians. In \Cref{sec:JR_on_surfaces}, we construct solutions to the Jackiw--Rossi equations on closed surfaces and curved planes. \Cref{sec:applications} is devoted to four applications of \Cref{theorem:general}. For completeness, in \Cref{sec:pairings}, we summarize the necessary notation and background from the representation theory of Clifford algebras and modules. In \Cref{app:BdG}, we also construct generalizations of the Bogoliubov--de Gennes equation.

\bigskip

\section{Generalized Jackiw--Rossi theory on pseudo-Riemannian manifolds}
\label{sec:JR_on_manifolds}

In this section, we generalize the above construction to bundles of Clifford modules. Let $\( X, g \)$ be a $d$-dimensional, smooth, oriented, pseudo-Riemannian manifold of signature $(t, s)$. As before, let $r = t - s$. Let $\CL \( T^* X, g \)$ be the bundle of (real) Clifford algebras given by
\begin{equation}
	\forall x \in X : \quad \CL \( T^* X, g \)_x \eqdef \CL \( T_x^* X, g_x \).
\end{equation}
Finally, let $\E \rightarrow X$ be a Hermitian vector bundle with a Clifford structure. In other words, there is a Clifford multiplication, $\cl : \CL \( T^* X, g \) \rightarrow \End \( \E \)$, that is a unital homomorphism of algebras, such that for all $v \in T^* X$, $\cl (v) \in \su \( \E \)$. We can think of $\cl$ as a map from $\bigwedge^* (X) \otimes \E$ to $\E$. The associated Hermitian vector bundle, $\cF$, is defined via
\begin{equation}
	\E \otimes \E \cong \left\{
	\begin{array}{ll}
		\bigwedge_\cx^{\mathrm{even}} (X) \otimes \cF,	& \mbox{if $d$ is odd,} \\
		\bigwedge_\cx^* (X) \otimes \cF,				& \mbox{if $d$ is even}.
	\end{array} \right. \label{eq:def_cF}
\end{equation}
We remark that the isomorphism above is canonical; cf. \cite{N07}*{Proposition 11.1.27.}. Thus, in particular, we again get a pairing $\cB_\E : \E \otimes \E \rightarrow \cF$. Furthermore, $\cF^*$ decomposes as $\cF^* \cong \cF^+ \oplus \cF^-$, to symmetric and anti-symmetric parts, as in \Cref{sec:CL_modules}. Note that $\cF$ need not be a tensor-square of another vector bundle.

We call a connection on $\E$ \emph{compatible}, if it is unitary and the Clifford multiplication is parallel. Let $\( \nabla, \Phi \)$ be a pair of a compatible connection $\nabla$ and a smooth section of $\cF^*$, $\Phi$. We define two operators on sections of $\E$ as follows:
\begin{enumerate}

	\item Using $\nabla$, we define the twisted Dirac operator as
	\begin{equation}
		\D_\nabla : L_1^2 \( \E \) \rightarrow L^2 \( \E \); \: \Psi \mapsto \D_\nabla \Psi = \cl \( \nabla \Psi \).
	\end{equation}

	\item Using $\Phi$, we define the perturbation term $\cA_\Phi$ as
	\begin{equation}
	 	h^\E \( \cA_\Phi \Psi_1, \Psi_2 \) \eqdef \Phi \( \cB_\E \( \Psi_1, \Psi_2 \) \). \label{eq:A_Phi}
	\end{equation}
	If $\Phi$ is a section of $\cF^\pm$, then let $\s_\Phi \eqdef \pm \s_r$ and, thus, $\cA_\Phi^* = \s_\Phi \cA_\Phi$.

\end{enumerate}

\begin{example}
	Let $X$ be spin and $\E = \S$ a spinor bundle corresponding to a spin structure. Then $\cF^*$ is the trivial line bundle. When equipped with the product connection, $\D = \D_\nabla$ is just the ordinary Dirac operator. If $\Phi$ is a covariantly constant and unit length, then $\cA = \cA_\Phi$ is just the canonical (up to scalar) isomorphism from $\S$ to its conjugate. Similarly, if $X$ is spin$^{\mathrm{c}}$, then $\cF^*$ is the associated line bundle, and a connection and a section of $\cF^*$ provides further examples of the above operators. In both cases $\cF^* = \cF^+$.
\end{example}

\begin{example}
	Let $\S$ be a spinor bundle and $E$ be a real or a quaternionic vector bundle and $\E = \S \otimes E$. Then $\cF^* = E^* \otimes E^*$, and the real or quaternionic structure defines a section, $\Phi_0$, of $\cF^*$. In fact $\Phi_0$ takes values in $\cF^+$ in the case of a real structure, and in $\cF^-$ in the case of a quaternionic structure. Thus, $\s_{\Phi_0} = \pm \s_r$, accordingly.
\end{example}

A particular case of the above examples are known in physics as the \emph{Jackiw--Rossi theory}; cf. \cite{CJNPS10}.

\begin{example}
	Let $\Sigma$ be an oriented, Riemannian surface, $\cL$ a Hermitian line bundle over $\Sigma$, and $\Theta$ be a ``square root'' of the canonical line bundle of $\Sigma$. Then $\S_\cL^\pm = \cL \Theta^{\pm 1}$ defines a spin$^{\mathrm{c}}$ spinor bundle, and the corresponding line bundle is $\cF = \cF^+ = \cL^{- 2}$.

	For each compatible connection, $\nabla$, on $\cF$, we get a connection on $\S_\cL$ and, thus, a twisted Dirac operator
	\begin{equation}
		\D_\nabla \begin{pmatrix} \Psi^+ \\ \Psi^- \end{pmatrix} \eqdef \begin{pmatrix} \delbar_\nabla^* \Psi^- \\ \delbar_\nabla \Psi^+ \end{pmatrix}.
	\end{equation}
	For each section, $\Phi$, of $\cF$, we also get an algebraic operator
	\begin{equation}
		\cA_\Phi \begin{pmatrix} \Psi^+ \\ \Psi^- \end{pmatrix} \eqdef \begin{pmatrix} \Phi \( \overline{\Psi^-} \) \\ \Phi \( \overline{\Psi^+} \) \end{pmatrix}.
	\end{equation}
	The operator $\D_\nabla + \cA_\Phi$ is the (real, symmetric) Jackiw--Rossi operator in \cite{CJNPS10} and for each $\mu \in \rl$ the functional
	\begin{equation}
		\Gamma \( \E \) \rightarrow \rl; \quad \Psi \mapsto h^\S \( \Psi, \D_\nabla \Psi \) + \Re \( h^\S \( \cA_\Phi \Psi, \Psi \) \) - \tfrac{\mu}{2} |\Psi|^2,
	\end{equation}
	is the energy density in \cite{CJNPS10}*{Equation~(1)}.

	Let now $X \eqdef \rl_t \times \Sigma$ be equipped with $g_X \eqdef - \rd t^2 + g$, to make it a Lorentzian 3-manifold. The pullback of $(\E, h, \cl)$ canonically defines a spin$^{\mathrm{c}}$ spinor bundle with
	\begin{equation}
		\cl \( \rd t \) = \begin{pmatrix} \id_{\S_\cL^+} & 0 \\ 0 & - \id_{\S_\cL^-} \end{pmatrix},
	\end{equation}
	which is the (pullback of the) parity operator. Using a slight abuse of notation, let $\( \nabla, \Phi \)$ denote also the pullback fields to $X$. The corresponding Jackiw--Rossi equation can then be described as the following flow:
	\begin{align}
		\nabla_{\del_t} \Psi^+	&= H_{\nabla, \Phi}^- \Psi^-, \\
		\nabla_{\del_t} \Psi^-	&= - H_{\nabla, \Phi}^+ \Psi^+.
	\end{align}
	which is equivalent to \cite{jackiw_zero_1981}*{Equation~(2.8)}.
\end{example}

The last example is the central motivation for this paper, in general, and for the next definition, in particular.

\begin{definition}\label{definition:JR_op}
	Let $\( X, g \)$ be an oriented, Riemannian $d$-manifold, and let $\( \E, h, \cl \)$ be a bundle of Hermitian Clifford modules and $\cF$ be the associated vector bundle. Furthermore, let $\( \nabla, \Phi \)$ be a pair as above. Then the \emph{(generalized) Jackiw--Rossi operator}, associated to the data $\( X, g, \E, h, \cl, \nabla, \Phi \)$, is
	\begin{equation}
		H_{\nabla, \Phi} \eqdef \D_\nabla + \cA_\Phi : L_1^2 \( \E \) \rightarrow L^2 \( \E \). \label{eq:JR_op}
	\end{equation}
	Similarly, we call the equation
	\begin{equation}
		H_{\nabla, \Phi} \Psi = 0, \label{eq:JR_eq}
	\end{equation}
	the \emph{Jackiw--Rossi equation}.
\end{definition}

\smallskip

In the following lemma we prove that any conjugate linear perturbation of a Dirac-type operator can be viewed as (\emph{half} of) a generalized Jackiw--Rossi operator.

\begin{lemma}
	\label{lemma:universality}
	Let $E_1$ and $E_2$ be Hermitian vector bundles over a smooth, closed manifold, $X$. Let
	\begin{equation}
		D : \Gamma \( E_1 \) \rightarrow \Gamma \( E_2 \),
	\end{equation}
	be a Dirac-type operator, and
	\begin{equation}
		A : \Gamma \( E_1 \) \rightarrow \Gamma \( E_2 \),
	\end{equation}
	be a conjugate linear bundle map.

	Then $\E \eqdef E_1 \oplus E_2$ is a Hermitian Clifford module bundle, where the Clifford multiplication, $c$, is induced by the symbol of $D$. Furthermore, let
	\begin{equation}
		H \eqdef \begin{pmatrix} 0 & D^* + A^* \\ D + A & 0 \end{pmatrix}, \label{eq:H_def}
	\end{equation}
	is a generalized Jackiw--Rossi operator as in \Cref{definition:JR_op}, with the connection, $\nabla$, defined uniquely by $D$, and the values of $\Phi$ are given by \Cref{lemma:uniqueness} through the bilinear maps
	\begin{equation}
		\forall x \in X : \quad \cB_x \( \begin{pmatrix} \Psi_{E_1, 1} \\ \Psi_{E_2, 1} \end{pmatrix}, \begin{pmatrix} \Psi_{E_1, 2} \\ \Psi_{E_2, 2} \end{pmatrix} \) \eqdef h_x^{E_1} \( A_x \Psi_{E_1, 1}, \Psi_{E_2, 2} \) + h_x^{E_1} \( A_x \Psi_{E_1, 2}, \Psi_{E_2, 1} \). \label{eq:induced_pairing}
	\end{equation}
	Finally,
	\begin{align}
		\ker \( H \)			&= \ker \( D + A \) \oplus \ker \( \( D + A \)^* \), \\
		\Spec \( H \) - \{ 0 \}	&= \left\{ \: \lambda \in \rl \: \middle| \: \lambda^2 \in \Spec \( \( D + A \)^* \( D + A \) \) - \{ 0 \} \: \right\}, \\		
		\forall \lambda \in \Spec \( H \) - \{ 0 \} : \quad \ker \( H - \lambda \id \) &= \left\{ \: \begin{pmatrix} \lambda \Psi \\ \( D + A \) \Psi \end{pmatrix} \: \middle| \: \Psi \in \ker \( \( D + A \)^* \( D + A \) - \lambda^2 \id \) \: \right\}.
	\end{align}
\end{lemma}

\begin{proof}
	Clearly, $\E$ is a Clifford module bundle. Let $\cF$ defined through \cref{eq:def_cF}. Let $\nabla^1$ be a unitary connection induced by $D$, that is a unitary connection on $E_1$, such that $D = c \circ \nabla^1$. Such a connection exists and is unique, as $D$ is Dirac-type. Similarly, let $\nabla^2$ induced by $D^*$, and let $\nabla \eqdef \nabla^1 \oplus \nabla^2$. Then the twisted Dirac operator on $\E$ is
	\begin{equation}
		\D_\nabla = \begin{pmatrix} 0 & D^* \\ D & 0 \end{pmatrix}.
	\end{equation}
	Similarly, if $\Phi$ is given via \cref{eq:induced_pairing}, then
	\begin{equation}
		\cA_\Phi = \begin{pmatrix} 0 & A^* \\ A & 0 \end{pmatrix},
	\end{equation}
	and, thus, $H = \D_\nabla + \cA_\Phi$ is of the form \eqref{eq:H_def}, and hence $H$ is a generalized Jackiw--Rossi operator. The claims about the spectra and eigenspinors are then straightforward.
\end{proof}

\smallskip

For each $r$, let us define $\sigma_r, \s_r \in \{ - 1, 1 \}$ as
\begin{equation}
	\cA \circ \cl (\cdot) = \sigma_r \cl (\cdot) \circ \cA, \quad \& \quad \cA^2 = \s_r \id_\S. \label{eq:sigma_s_def}
\end{equation}
In even dimensions, where $\sigma_r$ is also a choice, our conventions are
\begin{equation}
	\sigma_0 = - 1, \: \sigma_2 = - 1, \: \sigma_4 = 1, \: \sigma_6 = 1.
\end{equation}

Let us now prove a technical result about Jackiw--Rossi operators.

\begin{lemma}\label{lemma:Concentrating_Pair}
	Let $H_{\nabla, \Phi}$ be as in \Cref{definition:JR_op}, and let $\cl \( \cA_{\nabla \Phi} \) \eqdef \sum_{i = 1}^d \cl \( \rd x^i \) \cA_{\nabla_i \Phi}$. Then we have that
	\begin{equation}
		\D_\nabla \circ \cA_\Phi = \cl \( \cA_{\nabla \Phi} \) + \sigma_r \cA_\Phi \circ \D_\nabla.
	\end{equation}
	Thus, if $\Phi$ takes values only in $\cF^\pm$, then
	\begin{equation}
		\D_\nabla^* \circ \cA_\Phi + \cA_\Phi^* \circ \D_\nabla = \cl \( \cA_{\nabla \Phi} \) + (\sigma_r \pm \s_r) \cA_\Phi^* \circ \D_\nabla.
	\end{equation}
\end{lemma}

\begin{proof}
	Since the claims are local, we verify them at an arbitrary point of $X$, using a normal chart center at that point.
	
	Since $\D_\nabla$ is self adjoint and $\cA_\Phi^* = \pm \s_r \cA_\Phi$, we have
	\begin{align}
		\D_\nabla^* \circ \cA_\Phi + \cA_\Phi^* \circ \D_\nabla	&= \D_\nabla \circ \cA_\Phi \pm \s_r \cA_\Phi \circ \D_\nabla \\
																&= \sum\limits_{i = 1}^d \( \cl \( \rd x^i \) \circ \nabla_i \circ \cA_\Phi \pm \s_r \cA_\Phi \circ \cl \( \rd x^i \) \circ \nabla_i \) \\
																&= \sum\limits_{i = 1}^d \( \cl \( \rd x^i \) \circ (\nabla_i \cA_\Phi) + \cl \( \rd x^i \)\circ \cA_\Phi \circ \nabla_i \pm \s_r \cA_\Phi \circ \cl \( \rd x^i \) \circ \nabla_i \) \\
																&= \sum\limits_{i = 1}^d \( \cl \( \rd x^i \) \circ \cA_{\nabla_i \Phi} + (\sigma_r \pm \s_r) \cA_\Phi \circ \cl \( \rd x^i \) \circ \nabla_i \) \\
																&= \cl \( \cA_{\nabla \Phi} \) + (\sigma_r \pm \s_r) \cA_\Phi^* \circ \D_\nabla
	\end{align}
	which concludes the proof.
\end{proof}

Using \Cref{lemma:Concentrating_Pair} we prove a concentration property for Jackiw--Rossi operators.

\begin{theorem}\label{theorem:Manos}
	Let $X$ now be closed, $\pm \eqdef - \sigma_r \s_r$ (that is, $\sigma_r \pm \s_r = 0$), $\Phi$ be a section $\cF^\pm$, $H_{\nabla, \Phi}$ be as in \Cref{definition:JR_op}, and $(t_n)_{n \in \N}$ be a sequence of real numbers that converges to infinity. Let $\cZ_\Phi$ be the degenerate locus of $\Phi$, that is
 	\begin{equation}
		\cZ_\Phi \eqdef \left\{ \: x \in X \: \middle| \: \left\{ 0_x \right\} \neq \ker \( \cA_{\Phi_x} \) \leqslant \E_x \: \right\}.
	\end{equation}
	Assume that for each $n \in \N$ we have an eigenspinors of $H_{\nabla, t_n \Phi}$, call $\Psi_n \in \Gamma \( \E \)$, and
	\begin{equation}
		B \eqdef \limsup_{n \in \N} \frac{ \| H_{\nabla, t_n \Phi} \Psi_n \|_{L^2 \( X, g \)}^2}{t_n \| \Psi_n \|_{L^2 \( X, g \)}^2} < \infty.
	\end{equation}
	Then, for all positive integers $k, l$ there are constants $C = C (X, g, \E, h, \cl, \nabla, \Phi, B, k, l) > 0$, such that if $\Omega_n \subseteq X$ is a sequence of open sets with $\dist (\Omega_n, \cZ_\Phi) > 0$, then
	\begin{equation}
		\| \Psi_n \|_{L_k^2 (\Omega_n)} \leqslant \frac{C}{(\dist (\Omega_n, \cZ_\Phi)^2 t_n)^l} \| \Psi_n \|_{L^2 (X)}. \label[ineq]{ineq:manos_bound}
	\end{equation}
\end{theorem}

\begin{proof}
	The parity of $\Phi$ is chosen, using \Cref{lemma:Concentrating_Pair}, so that $\cl$ and $\cA_\Phi$ form a \emph{concentrating pair}, in the sense of Maridakis; cf. \cite{maridakis_spinor_2015}*{Definition~2.1}. Then the result is a special case of \cite{maridakis_spinor_2015}*{Corollary~2.6}.
\end{proof}

The final lemma of this section shows that while the theory is not conformally invariant, its kernel is conformally \emph{equivariant}. The proof of this lemma is mostly a generalization of \cite{lawson_spin_1989}*{Theorem~5.24} and \cite{charbonneau_spatially_2006}*{Appendix~D}. Note, however, that \cite{lawson_spin_1989}*{Theorem~5.24} had an incorrect factor in \cite{lawson_spin_1989}*{Equation~(5.39)} that is fixed in \cite{charbonneau_spatially_2006}*{Equation~(D.1)}.

\begin{lemma}
\label{lemma:conformal}
	Let $\( X, g, \E, h, \cl, \nabla, \Phi \)$ is as in \Cref{definition:JR_op}. Let $H_{\nabla, \Phi}$ be the corresponding Jackiw--Rossi operator. Let $\varphi$ be a smooth, real-valued function on $X$ and $d \eqdef \dim_\rl (X)$.

	Conformally changed Riemannian manifold, $\( X, g^\varphi \eqdef e^{2 \varphi} g \)$, has a Hermitian Clifford module bundle $\( \E^\varphi, h^\varphi, \cl^\varphi \)$ that is canonically isomorphic, as a smooth vector bundle, to $\E$ via \cite{lawson_spin_1989}*{Equation~(5.37)}, and moreover, under this isomorphism we have
	\begin{equation}
		h^\varphi = h, \quad \& \quad \cl^\varphi = e^{- \varphi} \cl. \label{eq:clifford_str_conformal_change}
	\end{equation}
	In particular, this is an isomorphism of Hermitian structures. The Levi-Civita connection transforms as
	\begin{equation}
		\nabla^\varphi \eqdef \nabla + \tfrac{1}{4} \left[ \cl \( \rd \varphi \), \cl (\cdot) \right].
	\end{equation}
	Then the Jackiw--Rossi operator corresponding to $\( X, g^\varphi, \E^\varphi, h^\varphi, \cl^\varphi, \nabla^\varphi, e^\varphi \Phi \)$ satisfies, under the above mentioned canonical isomorphism, that
	\begin{equation}
		H_{\nabla^\varphi, e^\varphi \Phi} = e^{- \frac{d + 1}{2} \varphi} \circ H_{\nabla, \Phi} \circ e^{\frac{d - 1}{2} \varphi}. \label{eq:JR_conformal_change}
	\end{equation}
	In particular, there is a map
	\begin{equation}
		\ker \( H_{\nabla, \Phi} \) \rightarrow \ker \( H_{\nabla^\varphi, e^\varphi \Phi} \); \: \Psi \mapsto e^{- \frac{d - 1}{2} \varphi} \Psi, \label{eq:kernel_conformal_change}
	\end{equation}
	which is an isomorphism.
\end{lemma}

\begin{proof}
	The statements about $\( \E^\varphi, h^\varphi, \cl^\varphi \)$, in particular \cref{eq:clifford_str_conformal_change} are standard. The statement about $\nabla^\varphi$ follows from \cite{lockhart_elliptic_1985}*{page~110}. Thus, we have
	\begin{align}
		e^{\frac{d + 1}{2} \varphi} \circ \( \sum\limits_{i = 1}^d \cl^\varphi \( \rd x^i \) \nabla_{\del_i}^\varphi \) \circ e^{- \frac{d - 1}{2} \varphi}	&= e^{\frac{d + 1}{2} \varphi} \circ \( \sum\limits_{i = 1}^d e^{- \varphi} \cl \( \rd x^i \) \( \nabla_{\del_i} + \tfrac{1}{4} \left[ \cl \( \rd \varphi \), \cl \( \rd x^i \) \right] \) \) \circ e^{- \frac{d - 1}{2} \varphi} \\
					&= \sum\limits_{i = 1}^d \cl \( \rd x^i \) e^{\frac{d - 1}{2} \varphi} \circ \nabla_{\del_i} \circ e^{- \frac{d - 1}{2} \varphi} \\
					&\quad + \tfrac{1}{4} \sum\limits_{i = 1}^d \cl \( \rd x^i \) \left[ \cl \( \rd \varphi \), \cl \( \rd x^i \) \right] \\
					&= \sum\limits_{i = 1}^d \( \cl \( \rd x^i \) \nabla_{\del_i} - \tfrac{d - 1}{2} \del_i \varphi \cl \( \rd x^i \) \) \\
					&\quad + \tfrac{1}{4} \sum\limits_{i = 1}^d \cl \( \rd x^i \) \left[ \cl \( \rd \varphi \), \cl \( \rd x^i \) \right] \\
					&= \D_\nabla - \tfrac{d - 1}{2} \cl \( \rd \varphi \) + \tfrac{1}{4} \sum\limits_{i = 1}^d \cl \( \rd x^i \) \left[ \cl \( \rd \varphi \), \cl \( \rd x^i \) \right].
	\end{align}
	Let us pick a local normal chart for $g$. Then
	\begin{align}
		\sum\limits_{i = 1}^d \cl \( \rd x^i \) \left[ \cl \( \rd \varphi \), \cl \( \rd x^i \) \right]	&= \sum\limits_{i, j = 1}^d \del_j \varphi \cl \( \rd x^i \) \left[ \cl \( \rd x^j \), \cl \( \rd x^i \) \right] \\
		&= 2 \sum\limits_{i, j = 1, i \neq j}^d \del_j \varphi \cl \( \rd x^j \) \\
		&= 2 (d - 1) \cl \( \rd \varphi \),
	\end{align}
	and thus
	\begin{equation}
		\sum\limits_{i = 1}^d \cl^\varphi \( \rd x^i \) \nabla_{\del_i}^\varphi = e^{- \frac{d + 1}{2} \varphi} \circ \D_\nabla \circ e^{\frac{d - 1}{2} \varphi},
	\end{equation}
	which proves the claim for the twisted Dirac operator part of $H_{\nabla, \Phi}$. For the algebraic part note $\cA_\Phi$ is linear in $\Phi$, thus
	\begin{equation}
		\cA_{e^\varphi \Phi} = e^\varphi \cA_\Phi = e^{- \frac{d + 1}{2} \varphi} \circ \cA_\Phi \circ e^{\frac{d - 1}{2} \varphi},
	\end{equation}
	which concludes the proof of \cref{eq:JR_conformal_change}. The proof of \cref{eq:kernel_conformal_change} is now trivial.
\end{proof}

\bigskip

\section{The 2-dimensional Jackiw--Rossi theory}

\subsection{The model case}
\label{sec:model}

In this section we analyze a simple case which serves as the model and main ingredient for the results of the next section.

Consider $\cx$ with its flat metric and let $\dA$ be its area form. Let $\alpha$ be a smooth, complex valued, compactly supported function on $\cx$, and $\Phi$ be a homogeneous, complex polynomial with an isolated zero at the origin of the form
\begin{equation}
	\Phi (z, \overline{z}) = T z^p \overline{z}^q + \varphi (z, \overline{z}),
\end{equation}
where $m \eqdef p + q$ is the degree of $\Phi$, and for some $\epsilon \in (0, |T|)$, it satisfies
\begin{equation}
	|\varphi (z, \overline{z})| \leqslant \( |T| - \epsilon \) |z|^m.
\end{equation}
The index of $\Phi$ at zero is $k \eqdef p - q \in \Z$.

For $f \in L^2 (\cx)$, let us consider (the weak formulation of) the equation:
\begin{equation}
	\frac{\del f}{\del \overline{z}} + \alpha f + \Phi \overline{f} = 0. \label{eq:model_eq}
\end{equation}

The main result of this section is the following theorem.

\begin{theorem}\label{theorem:model}
	There is $c = c (\epsilon) > 0$, such that if $\| \alpha \|_{L^\infty (\cx)} \leqslant c \sqrt[m + 1]{|T|}$, then the dimension of the space of solutions to \cref{eq:model_eq} is $\max \( \{ k, 0 \} \)$. Moreover, solutions to \cref{eq:model_eq} decay exponentially.
\end{theorem}

\begin{remark}
	The case when $\alpha = 0$ and $\varphi = 0$ was considered by Rauch in \cite{R04}. However for the purposes of this paper we need to assume that neither $\alpha$, nor $\varphi$ vanish, and these requirements pose some nontrivial technical difficulties.
\end{remark}

Before the proof of \Cref{theorem:model}, we prove a few technical results below.

\smallskip

First we consider the $\alpha = 0$, but $\varphi \neq 0$ case. For each $f \in C_{\mathrm{cpt}}^\infty (\cx)$, let
\begin{align}
	\widetilde{D}^+ f	&\eqdef \frac{\del f}{\del \overline{z}} + \Phi \overline{f}, \\
	\widetilde{D}^- f	&\eqdef - \frac{\del f}{\del z} + \Phi \overline{f}.
\end{align}
These define discontinuous operators on $L^2 (\cx)$ with dense domains, that,  for any $f, g \in C_{\mathrm{cpt}}^\infty (\cx)$, satisfy
\begin{equation}
	\Re \( \langle \widetilde{D}^+ f | g \rangle_{L^2 \( \cx \)} \) = \Re \( \langle f | \widetilde{D}^- g \rangle_{L^2 \( \cx \)} \),
\end{equation}
and the following Weitzenb\"ock-type identities:
\begin{subequations}
\begin{align}
	\widetilde{D}^- \widetilde{D}^+ f &= \Delta f + |\Phi|^2 f - (\del \Phi) \overline{f}, \label{eq:weitz1a} \\
	\widetilde{D}^+ \widetilde{D}^- f &= \Delta f + |\Phi|^2 f + (\delbar \Phi) \overline{f}, \label{eq:weitz1b}
\end{align}
\end{subequations}
and thus
\begin{subequations}
\begin{align}
	\| \widetilde{D}^+ f \|_{L^2 (\cx)}^2 &= \| \rd f \|_{L^2 (\cx)}^2 + \| \Phi f \|_{L^2 (\cx)}^2 - \int\limits_{\cx} (\del \Phi) \overline{f}^2 \dA, \label{eq:weitz2a} \\
	\| \widetilde{D}^- f \|_{L^2 (\cx)}^2 &= \| \rd f \|_{L^2 (\cx)}^2 + \| \Phi f \|_{L^2 (\cx)}^2 + \int\limits_{\cx} (\delbar \Phi) \overline{f}^2 \dA. \label{eq:weitz2b}
\end{align}
\end{subequations}
Furthermore, since their duals are also densely defined, both $\widetilde{D}^\pm$ are closeable, in fact $D^\pm = \( \widetilde{D}^\mp \)^*$ are the unique closed extensions of $\widetilde{D}^\pm$. Let $\cH^\pm = \dom \( D^\pm \)$ equipped with the graph norm (which defines a real Hilbert space structure).

Now we are ready to prove the following theorem, which a fortiori implies \Cref{theorem:model} in the $\alpha = 0$ case.

\begin{theorem}\label{theorem:model_analysis}
	Both $\widetilde{D}^\pm$ have unique closed extensions, $D^\pm$. Let $\cH^\pm \eqdef \dom \( D^\pm \)$ equipped with the graph norm (which defines a real Hilbert space structure). The operators $D^\pm : \cH^\pm \rightarrow L^2 (\cx)$ are Fredholm and
	\begin{equation}
		\dim_\rl \( \ker \( D^\pm \) \) = \max \( \{ \pm k, 0 \} \). \label{eq:dimker}
	\end{equation}
	In particular, the Fredholm index of $D^\pm$ is $\pm k$.

	Furthermore, elements of $\ker \( D^\pm \)$ decay exponentially.
\end{theorem}

\begin{proof}
	In order to show that $\widetilde{D}^\pm$ is closeable, we need to prove that if $\( f_n^\pm \)_{n \in \N}$ is a sequence in $C_{\mathrm{cpt}}^\infty (\cx)$, such that $f_n^\pm \rightarrow 0$ in $L^2 (\cx)$, and $D^\pm f_n^\pm \rightarrow g^\pm$ also in $L^2 (\cx)$, then $g^\pm = 0 \in L^2 (\cx)$. Let $\chi \in C_{\mathrm{cpt}}^\infty (\cx)$ be arbitrary and fixed. Then
	\begin{equation}
		\langle \chi | D^\pm f_n^\pm \rangle_{L^2 (\cx)} = \langle D^\mp \chi | f_n^\pm \rangle_{L^2 (\cx)}.
	\end{equation}
	The left-hand side converges to $\langle \chi | g \rangle_{L^2 (\cx)}$, while the right-hand side converges to zero, since $|D^\mp \chi| \in C_{\mathrm{cpt}}^\infty (\cx)$. Thus $g = 0$. Hence both $\widetilde{D}^\pm$ are closeable, and since they are densely defined, the closed extensions, $D^\pm$, are unique. Note that \cref{eq:weitz2a,eq:weitz2b} imply that $\cH^\pm = \dom \( D^\pm \) \leqslant L_1^2 (\cx)$.

	Next, we prove that all $f \in \ker \( D^\pm \)$ are smooth. We present the proof for $f \in \ker \( D^+ \)$, as the two cases are analogous. Since $f \in \cH^+ = \dom \( D^+ \) \leqslant L_1^2 (\cx) \leqslant L_{1, \mathrm{loc}}^2 (\cx)$ and $\Phi$ is smooth, we have that
	\begin{equation}
		\frac{\del f}{\del \overline{z}} = - \Phi \overline{f} \in L_{1, \mathrm{loc}}^2 (\cx),
	\end{equation}
	and, thus, by usual elliptic regularity and bootstrap, $f$ is smooth.

	Next, we show that each $f \in \ker \( D^\pm \)$ decays exponentially. We again only show for the $+$ case. By \cref{eq:model_eq} and the smoothness of $f$ we have that
	\begin{equation}
		\Delta f = - \frac{\del f}{\del z \del \overline{z}} = \frac{\del \Phi}{\del z} \overline{f} - |\Phi|^2 f,
	\end{equation}
	and thus
	\begin{equation}
		\( \tfrac{1}{2} \Delta + |\Phi|^2 - \left| \frac{\del \Phi}{\del z} \right| \) |f|^2 \leqslant 0.
	\end{equation}
	By the assumptions,for some $R$ large enough, $\Phi$ satisfies on $\cx - B_R (0)$ that
	\begin{equation}
		\( \Delta + 1 \) |f|^2 \leqslant 0,
	\end{equation}
	and, thus, $|f|$ decays exponentially as $|z| \rightarrow \infty$.

	Next, we show that the dimension of $\ker \( D^+ \)$ is at most $k$.  Since $0 \leqslant |k| \leqslant m$, we have that $p$ and $q$ are both nonnegative. Recall that $T \in \cx$ is the coefficient of $z^p \overline{z}^q$ in $\Phi$, and write $\varphi = \Phi - T z^p \overline{z}^q$. By the hypotheses on $\Phi$ we have that $\varphi$ is also a homogeneous, complex polynomial of degree $m$, and for some $\epsilon \in (0, |T|)$, it satisfies $|\varphi (z, \overline{z})| \leqslant |T| (1 - \epsilon) |z|^m$. Now we have
	\begin{equation}
		- \frac{1}{2} \frac{\del}{\del \overline{z}} \( \frac{f^2}{z^k} \) = - \frac{f \frac{\del f}{\del \overline{z}}}{z^k} = |f|^2 \frac{\Phi}{z^k} = |f|^2 \( T |z|^{2 q} + \frac{\varphi}{z^k} \) \geqslant \epsilon |z|^{2 q} |f|^2.
	\end{equation}
	Hence, using the exponential decay of $f$, we get
	\begin{align}
		\frac{\epsilon}{\pi} \int\limits_\cx |z|^{2 q} |f|^2 \dA	&\leqslant \lim\limits_{r \rightarrow 0^+} \lim\limits_{R \rightarrow \infty} \frac{- 1}{2 \pi} \int\limits_{B_R (0) - B_r (0)} \frac{\del}{\del \overline{z}} \( \frac{f^2}{z^k} \) \dA \\
															&= \lim\limits_{R \rightarrow \infty} \frac{i}{2 \pi} \int\limits_{S_R (0)} \frac{f^2}{z^k} \rd z - \lim\limits_{r \rightarrow 0^+} \frac{i}{2 \pi} \int\limits_{S_r (0)} \frac{f^2}{z^k} \rd z \\
															&= \lim\limits_{r \rightarrow 0^+} \frac{1}{2 \pi i} \int\limits_{S_r (0)} \frac{f^2}{z^k} \rd z. \label[ineq]{ineq:localization}
	\end{align}
	Note that when $q = 0$, this formula resembles the Samols' Localization Formula; cf. \cite{S92}. Since $f$ is smooth, when $k \leqslant 0$, then the right hand side is zero, hence $f \equiv 0$. If $k > 0$, then let us have the $k^{\mathrm{th}}$ order Taylor expansion of $f$
	\begin{equation}
		f (z, \overline{z}) = \sum_{a = 0}^k \sum_{b = 0}^a f_{ab} z^{a - b} \overline{z}^b + O \( |z|^{k + 1} \).
	\end{equation}
	Plugging this into \cref{ineq:localization}, we get that
	\begin{equation}
		\frac{\epsilon}{\pi} \int\limits_\cx |z|^{2 q} |f|^2 \dA \leqslant \lim\limits_{r \rightarrow 0^+} \frac{1}{2 \pi i} \int\limits_{S_r (0)} \frac{f^2}{z^k} \rd z = \sum_{a = 0}^{k - 1} f_{a0} f_{(k - 1 - a)0}. \label[ineq]{ineq:localization2}
	\end{equation}
	Hence the kernel of the linear map
	\begin{equation}
		\ker \( D^+ \) \ni f \mapsto \sum_{a = 0}^{k - 1} f_{a0} z^a,
	\end{equation}
	is trivial and, thus, the dimension of $\ker \( D^+ \)$ is at most $k$.

	Next, we show that there is $C^\pm = c_\epsilon^\pm \sqrt[m + 1]{|T|} > 0$, such that for all $f \in \( \ker \( D^\pm \) \)^\perp \leqslant \cH^\pm$:
	\begin{equation}
		\| D^\pm f \|_{L^2 (\cx)} \geqslant C^\pm \| f \|_{L^2 (\cx)}. \label[ineq]{ineq:gap}
	\end{equation}
	First, we show that there is some $C^\pm > 0$ satisfying \cref{ineq:gap}, using proof-by-contraction: Let assume that $\( f_n^\pm \)_{n \in \N}$ is a sequence in the unit sphere of $\( \ker \( D^\pm \) \)^\perp$ and $\( \epsilon_n \)_{n \in \N}$ is a sequence in $\rl_+$, such that
	\begin{equation}
		\lim\limits_{n \rightarrow \infty} \epsilon_n = 0,
	\end{equation}
	and for all $n \in \N$
	\begin{equation}
		\| D^\pm f_n^\pm \|_{L^2 (\cx)} \leqslant \epsilon_n \| f_n^\pm \|_{L^2 (\cx)}.
	\end{equation}
	Since the norm of $\cH^\pm$ is the graph norm of $D^\pm$, we get that
	\begin{equation}
		\| D^\pm f_n^\pm \|_{L^2 (\cx)} \leqslant \epsilon_n \| f_n^\pm \|_{L^2 (\cx)} \leqslant \epsilon_n \sqrt{\| D^\pm f_n^\pm \|_{L^2 (\cx)}^2 + \| f_n^\pm \|_{L^2 (\cx)}^2} \leqslant \epsilon_n.
	\end{equation}
	Thus, $\| D^\pm f_n^\pm \|_{L^2 (\cx)} \rightarrow 0$. Similarly $\| f_n \|_{L^2 (\cx)} \leqslant 1$. Let $K \subset \cx$ be the set on which $|\rd \Phi| \geqslant |\Phi|^2$. Note that $K$ is compact, since $\Phi$ is a polynomial. Using \cref{eq:weitz2a,eq:weitz2b} we then get that
	\begin{equation}
		\| \rd f_n^\pm \|_{L^2 (\cx)}^2 \leqslant \| D^\pm f_n^\pm \|_{L^2 (\cx)}^2 + \int\limits_{\cx} \max \( \left\{ |\rd \Phi| - |\Phi|^2, 0 \right\} \) |f_n^\pm|^2 \dA \leqslant \epsilon_n^2 + \| \rd \Phi \|_{L^\infty (K)}.
	\end{equation}
	Thus, $\( f_n^\pm \)_{n \in \N}$ is bounded in $L_1^2 (\cx)$, and hence has a subsequence that is convergent in $L^2 (\cx)$. Let us redefine $\( f_n^\pm \)_{n \in \N}$ to be that subsequence. Sicne $\| D^\pm f_n^\pm \|_{L^2 (\cx)} \rightarrow 0$, we have that $\( f_n^\pm \)_{n \in \N}$ is convergent in $\cH^\pm$ and converges to an element of $\ker \( D^\pm \)$, but this contradict the assumption that $\( f_n^\pm \)_{n \in \N}$ is a sequence in the unit sphere of $\( \ker \( D^\pm \) \)^\perp$. Thus, \cref{ineq:gap} holds for some $C^\pm > 0$. The dependence on $|T|$ follow immediately using standard rescalling arguments, using that $\Phi$ is homogenous of degree $m$ and rescalling $\cx$ by $z \mapsto \tfrac{z}{\sqrt[m + 1]{|T|}}$.

	Now we prove that both $D^\pm$ are Fredholm. We have already proved that they are continuous, and have finite dimensional kernels. It is then enough to prove that their cokernels are finite dimensional. We do that in two steps:
	\begin{enumerate}
		\item Prove that the images of $D^\pm$ are closed and, thus, $\coker \( D^\pm \) \cong \( \mathrm{im} \( D^\pm \) \)^\perp$.
		\item We show that $\( \mathrm{im} \( D^\pm \) \)^\perp = \ker \( D^\mp \)$.
	\end{enumerate}

	In order to prove (1), let $\Pi^\pm : \cH^\pm \rightarrow \ker \( D^\pm \)$ be the orthogonal projection, which is compact, as the kernels are finite dimensional (in fact, one of them is always trivial). Then, using \cref{ineq:gap}, for all $f \in \cH^\pm$, we have that
	\begin{align}
		\| f \|_{\cH^\pm}^2	&= \| D^\pm f \|_{L^2 (\cx)}^2 + \| f \|_{L^2 (\cx)}^2 \\
							&= \| D^\pm (f - \Pi^\pm f) \|_{L^2 (\cx)}^2 + \| (f - \Pi^\pm f) \|_{L^2 (\cx)}^2 + \| \Pi^\pm f \|_{L^2 (\cx)}^2 \\
							&\leqslant (1 + (C^\pm)^2) \| D^\pm (f - \Pi^\pm f) \|_{L^2 (\cx)}^2 + \| \Pi^\pm f \|_{L^2 (\cx)}^2 \\
							&= (1 + (C^\pm)^2) \| D^\pm f \|_{L^2 (\cx)}^2 + \| \Pi^\pm f \|_{L^2 (\cx)}^2.
	\end{align}
	Thus, by \cite{BS18}*{Lemma~4.3.9}, $\mathrm{im} \( D^\pm \)$ is closed.

	Finally, we show that $\( \mathrm{im} \( D^\pm \) \)^\perp = \ker \( D^\mp \)$. The fact that $\( \mathrm{im} \( D^\pm \) \)^\perp \geqslant \ker \( D^\mp \)$ is straightforward. In order to prove the other inclusion, note first that any $u \in L^2 (\cx)$ defines a continuous functional on $\cH^\pm$ by $f \mapsto \langle u | f \rangle_{L^2 (\cx)}$, because the norm on $\cH^\pm$ is strictly, pointwise greater than the one on $L^2 (\cx)$. Thus, by the Riesz Representation Theorem, there is a unique $v_u^\pm \in \cH^\pm$, such that for all $f \in \cH^\pm$, we have
	\begin{equation}
		\langle u | f \rangle_{L^2 (\cx)} = \langle v_u^\pm | f \rangle_{\cH^\pm} = \langle D^\pm v_u^\pm | D^\pm f \rangle_{L^2 (\cx)} + \langle v_u^\pm | f \rangle_{L^2 (\cx)}. \label{eq:weak_poisson}
	\end{equation}
	Note that \cref{eq:weak_poisson} is the weak formulation of the Poisson type equation
	\begin{equation}
		\( D^\mp D^\pm + 1 \) v_u^\pm = u. \label{eq:strong_poisson}
	\end{equation}
	Since $v_u^\pm \in \cH^\pm \leqslant L_1^2 (\cx)$, elliptic regularity guarantees that $v_u \in L_{2, \mathrm{loc}}^2 (\cx)$.

	Let now $u^\pm \in \( \mathrm{im} \( D^\pm \) \)^\perp$, or equivalently
	\begin{equation}
		\forall f \in \cH^\pm : \quad \langle u | D^\pm f \rangle_{L^2 (\cx)} = 0.
	\end{equation}
	Let $v_u^\pm$ as above, and define $w_u^\pm = D^\pm v_u^\pm \in L_{1, \mathrm{loc}}^2 (\cx)$. Now $w_u^\pm$ is a weak solution to
	\begin{equation}
		\( D^\pm D^\mp + 1 \) w_u^\pm = 0.
	\end{equation}
	Thus, $w_u^\pm$ is smooth, and hence so is $v_u^\pm$, again by elliptic regularity. This means, that $v_u^\pm$ is a strong solution to \cref{eq:strong_poisson} and, thus, $u$ is smooth, and standard argument shows that it satisfies $D^\mp u = 0$, and so $\( \mathrm{im} \( D^\pm \) \)^\perp \leqslant \ker \( D^\mp \)$.

	Now let us define
	\begin{equation}
		\Phi_t \eqdef T z^p \overline{z}^q + (1 - t) \varphi.
	\end{equation}
	This defines a homotopy between $\Phi_0 = \Phi$ and $\Phi_1 = T z^p \overline{z}^q$, such that for every $t \in [0, 1]$, $\Phi_t$ is a degree $m$ homogeneous polynomial with a unique zero of index $k$ at the origin. Furthermore, the corresponding operators, call $\left\{ D_t^\pm \right\}_{t \in [0, 1]}$ form a norm-continuous family of operators that, by the above arguments, are all Fredholm, and thus
	\begin{equation}
		\index_\rl \( D^\pm \) = \index_\rl \( D_0^\pm \) = \index_\rl \( D_1^\pm \).
	\end{equation}
	The index of $D_1^\pm$ was computed by Rauch in \cite{R04}*{Proposition~3.1} and is equal to $\pm k$. This, together with the triviality of $\ker \( D^{- \sign (k)} \)$, concludes the proof.
\end{proof}

Finally, we are ready to prove \Cref{theorem:model}.

\begin{proof}[Proof of \Cref{theorem:model}:]
	First, let $\alpha \in C_{\mathrm{cpt}}^\infty (\cx)$ arbitrary. The operators defined as
	\begin{align}
		D_\alpha^+ f	&= \frac{\del f}{\del \overline{z}} + \alpha f + \Phi \overline{f}, \\
		D_\alpha^- f	&= - \frac{\del f}{\del z} + \overline{\alpha} f + \Phi \overline{f}.
	\end{align}
	are compact perturbations of $D^\pm$ and, thus, are themselves Fredholm operators and have the same Fredholm indices, that is $\pm k$. Again, it is enough to consider the $k \geqslant 0$ case. For any $f_1, f_2 \in C_{\mathrm{cpt}}^\infty (\cx)$ we have
	\begin{equation}
		\langle f_1 | D_\alpha^+ f_2 \rangle_{L^2 (\cx)} = \langle D_\alpha^- f_1 | f_2 \rangle_{L^2 (\cx)},
	\end{equation}
	and, thus, it is enough to show that the kernel of $D_\alpha^-$ is trivial. Assume that $f \in \cH^-$ and $D_\alpha^- f = 0$, or equivalently
	\begin{equation}
		D^- f = - \overline{\alpha} f. \label{eq:fixed_pt}
	\end{equation}
	Using \cref{ineq:gap}, we get that
	\begin{equation}
		\| f \|_{L^2 (\cx)} \leqslant \frac{1}{c^- (\epsilon) \sqrt[m + 1]{|T|}} \| D^- f \|_{L^2 (\cx)} = \frac{1}{c^- (\epsilon) \sqrt[m + 1]{|T|}} \| \overline{\alpha} f \|_{L^2 (\cx)} \leqslant \frac{\| \alpha \|_{L^\infty (\cx)}}{c^- (\epsilon) \sqrt[m + 1]{|T|}} \| f \|_{L^2 (\cx)}.
	\end{equation}
	Thus, if $\| \alpha \|_{L^\infty (\cx)} < c^- (\epsilon) \sqrt[m + 1]{|T|}$, then $f = 0$. The claim about the exponential decay of the solutions can be proven as in the proof of \Cref{theorem:model_analysis}, as $\alpha$ is compactly supported.
\end{proof}

\bigskip

\section{Generalized Jackiw--Rossi theory on complete surfaces}
\label{sec:JR_on_surfaces}

In this section, we use the results of \Cref{sec:model} to classify and construct solutions to Jackiw--Rossi \cref{eq:JR_eq}.

Let $X$ be a oriented, complete Riemannian surface, and let $\S_\cL$ be a spinor bundle corresponding to a spin$^{\mathrm{c}}$ structure of $X$ with associated line bundle $\cL^2$. Let $\nabla$ be a compatible connection on $\S_\cL$, and $\Phi$ be a smooth section of $\cF^+ = \cL^{- 2}$. Thus, we can define the Jackiw--Rossi operator $H_{\nabla, \Phi} = \D_\nabla + \cA_\Phi$ as before. Recall that in dimension two, every connection induces a holomorphic structure. We use the usual notations, $\del_\nabla \eqdef \nabla^{1, 0}$ and $\delbar_\nabla \eqdef \nabla^{0, 1}$. In particular, in the splitting $\S_\cL = \S_\cL^+ \oplus \S_\cL^-$ where $\S_\cL^\pm = \cL \Theta^{\pm 1}$ with $\Theta^2 = K_X$, we have
\begin{equation}
	\D_\nabla = \begin{pmatrix} 0 & \delbar_\nabla^* \\ \delbar_\nabla & 0 \end{pmatrix}.
\end{equation}

Let $\cZ_\Phi \eqdef \left\{ \: x \in X \: \middle| \: \Phi (x) = 0 \: \right\}$ be the zero locus of $\Phi$. Assume for the rest of the section the following:
\begin{enumerate}
	\item[\textbf{A1:}] The curvature of $\nabla$ is bounded.
	\item[\textbf{A2:}] $\nabla \Phi$ is bounded.
	\item[\textbf{A3:}] $\cZ_\Phi$ is finite and, thus, the number
	\begin{equation}
		\delta_\Phi \eqdef \tfrac{1}{4} \inf \( \left\{ \: \dist (x, y) \: \middle| \: x, y \in \cZ_\Phi \: \& \: x \neq y \: \right\} \), \label{eq:delta}
	\end{equation}
	is positive.
	\item[\textbf{A4:}] The number
	\begin{equation}
		c_\Phi \eqdef \inf \( \left\{ \: |\Phi (x)| \: \middle| \: x \in X \: \& \: \dist (x, \cZ_\Phi) > \delta_\Phi \: \right\} \), \label{eq:def_c_Phi}
	\end{equation}
	is positive\footnote{When $X$ is compact, then \cref{eq:def_c_Phi} is implied by Assumption \textbf{A2}, but when $X$ is noncompact, then it excludes, for example, the cases when $\Phi$ decays at ``infinity''.}.
	\item[\textbf{A5:}] For all $x \in \cZ_\Phi$, there is a homogeneous, complex polynomial on $\cx$, $\phi_x$, of degree $m (x) \in \Z_+$ with an isolated zero at the origin of index $k (x) \in \Z$, a local, holomorphic, normal chart $z$, and a local, normal trivialization of $\cL$, $\nu_x$, such that for all $y \in X$ close enough points we have
	\begin{subequations}
	\begin{align}
		\Phi (y)		&= \phi_x (z (y), \overline{z (y)}) \nu_x (y)^\flat \otimes \nu_x (y)^\flat + O \( |z (y)|^{m (x) + 1} \), \label{eq:Phi_hypoth} \\
		\nabla \Phi (y)	&= O \( |z (y)|^{m (x) - 1} \). \label{eq:nabla_Phi_hypoth}
	\end{align}
	\end{subequations}
	In other words, the leading order term of $\Phi$ at any of its zero is of the form that we studied in \Cref{sec:model}.
\end{enumerate}

As in \Cref{sec:model}, let us define the (real) Hilbert spaces $\cH_{\nabla, \Phi}^\pm$ as the norm-completions of the vector space of smooth, compactly supported sections of $\S_\cL^\pm$, with respect to the $L^2$-graph norm of $H_{\nabla, \Phi}^\pm$, that is:
\begin{equation}
	\forall \Psi \in C_{\mathrm{cpt}}^\infty \( \S_\cL^\pm \) : \: \| \Psi \|_{\cH_{\nabla, \Phi}^\pm} = \( \| H_{\nabla, \Phi} \Psi \|_{L^2 \( \S_\cL \)}^2 + \| \Psi \|_{L^2 \( \S_\cL \)}^2 \)^{\frac{1}{2}},
\end{equation}
and let $\cH_{\nabla, \Phi} \eqdef \cH_{\nabla, \Phi}^+ \oplus \cH_{\nabla, \Phi}^-$.

The next two lemmata are technical results that are necessary to prove our main theorem.

\smallskip

\begin{lemma}
\label{lemma:regularity}
	Under Assumptions \textbf{A1} through \textbf{A5}, we have the following:
	\begin{enumerate}
		\item $\cH_{\nabla, \Phi} \subseteq L_1^2 \( \S_\cL \)$ and the inclusion is a continuous embedding. If $\Phi$ is bounded, then $\cH_{\nabla, \Phi} = L_1^2 \( \S_\cL \)$.
		\item Considered as a densely defined operator on $L^2 \( \S_\cL \)$, $H_{\nabla, \Phi}$ (with domain being $\cH_{\nabla, \Phi}$) is self-adjoint.
	\end{enumerate}
\end{lemma}

\begin{proof}
	Simple computation yields the following Weitzenb\"ock identity:
	\begin{equation}
		\( H_{\nabla, \Phi} \)^* H_{\nabla, \Phi} \Psi = \nabla^* \nabla \Psi + \cl \( F_\nabla \) \Psi + \cl \( \cA_{\nabla \Phi} \) \Psi + |\Phi|^2 \Psi. \label{eq:weitz}
	\end{equation}
	Thus, if $\Psi$ is a smooth and compactly supported section on $\S_\cL$, then we have
	\begin{align}
		\| \nabla \Psi \|_{L^2 \( T^* X \otimes \S_\cL \)}^2	&= \| H_{\nabla, \Phi} \Psi \|_{L^2 \( \S_\cL \)}^2 - \langle \Psi | \cl \( F_\nabla \) \Psi \rangle_{L^2 \( \S_\cL \)} - \langle \Psi | \cl \( \cA_{\nabla \Phi} \) \Psi \rangle_{L^2 \( \S_\cL \)} - \| \cA_\Phi \Psi \|_{L^2 \( \S_\cL \)}^2 \\
														&\leqslant \| H_{\nabla, \Phi} \Psi \|_{L^2 \( \S_\cL \)}^2 + 2 \( \| \nabla \Phi \|_{L^\infty} + \| F_\nabla \|_{L^\infty} \) \| \Psi \|_{L^2 \( \S_\cL \)}^2.
	\end{align}
	Using Assumptions \textbf{A1} and \textbf{A2}, we get $\| \Psi \|_{L_1^2 \( \S_\cL \)} \leqslant C \( \nabla, \Phi \) \| \Psi \|_{\cH_{\nabla, \Phi}}$. Furthermore, if $\Phi$ is bounded, then the above equations also gives $\| \Psi \|_{\cH_{\nabla, \Phi}} \leqslant C^\prime \( \nabla, \Phi \) \| \Psi \|_{L_1^2 \( \S_\cL \)}$. This proves the first claim.

	As $H_{\nabla, \Phi}$ is symmetric when restricted to smooth, compactly supported sections, it has a self-adjoint extension, which is necessarily closed, but then this extension must be $H_{\nabla, \Phi}$.
\end{proof}

\smallskip

\begin{lemma}
\label{lemma:spectrum}
	Under the Assumptions \textbf{A1} through \textbf{A5}, for all $\cE_0 > 0$, there is a positive number $T$, only dependent on $(X, g, \cL, \nabla, \Phi, \cE_0)$, such that for all $t > T$, the intersection $\left[ - \cE_0, \cE_0 \right] \cap \Spec \( H_{\nabla, t \Phi} \)$ contains only eigenvalues.

	Furthermore, if $\lambda \in \left[ - \cE_0, \cE_0 \right] \cap \Spec \( H_{\nabla, t \Phi} \)$, then $\ker \( H_{\nabla, t \Phi} - \lambda \id_{L^2 \( \S_\cL \)} \)$ is a finite dimensional subspace of $C^\infty \( \S_\cL \) \cap L_2^2 \( \S_\cL \)$.
\end{lemma}

\begin{proof}
	Let $\chi$ be the nonnegative function that satisfies
	\begin{equation}
		\chi^2 = \max \( \left\{ \: c_\Phi^2 - |\Phi|^2, 0 \: \right\} \),
	\end{equation}
	where $c_\Phi > 0$ is as in \cref{eq:def_c_Phi}. By Assumptions \textbf{A3} and \textbf{A4}, $\chi$ is compactly supported, $c_\Phi^2 - \chi^2 \leqslant |\Phi|^2$, and, thus, the map
	\begin{equation}
		K : \cH_{\nabla, t \Phi} \rightarrow L^2 \( \S_\cL \); \: \Psi \mapsto \chi \Psi, 
	\end{equation}
	is compact. Then for all $\Psi \in \cH_{\nabla, t \Phi}$ we have
	\begin{align}
		c_\Phi^2 \| \Psi \|_{L^2 \( \S_\cL \)}^2	&= \int\limits_X \( c_\Phi^2 - \chi^2 \) |\Psi|^2 \vol_X + \int\limits_X \chi^2 |\Psi|^2 \vol_X \\
			&\leqslant \| \Phi \Psi \|_{L^2 \( \cL^{- 2} \S_\cL \)}^2 + \| \chi \Psi \|_{L^2 \( \S_\cL \)}^2 \\
			&= \| \cA_\Phi \Psi \|_{L^2 \( \S_\cL \)}^2 + \| K \Psi \|_{L^2 \( \S_\cL \)}^2.
	\end{align}
	Using the Weitzenb\"ock identity from the proof of \Cref{lemma:regularity}, we get for any $\lambda \in \rl$ that
	\begin{align}
		t^2 \| \cA_\Phi \Psi \|_{L^2 \( \S_\cL \)}^2	&\leqslant \| H_{\nabla, t \Phi} \Psi \|_{L^2 \( \S_\cL \)}^2 + \( 2 t \| \nabla \Phi \|_{L^\infty} + 2 \| F_\nabla \|_{L^\infty} \) \| \Psi \|_{L^2 \( \S_\cL \)}^2 \\
		&= \| \( H_{\nabla, t \Phi} - \lambda \id_{L^2 \( \S_\cL \)} \) \Psi \|_{L^2 \( \S_\cL \)}^2 + 2 \lambda \langle \Psi | \( H_{\nabla, t \Phi} - \lambda \id_{L^2 \( \S_\cL \)} \) \Psi \rangle_{L^2 \( \S_\cL \)} \\
		& \quad + \( 2 t \| \nabla \Phi \|_{L^\infty} + 2 \| F_\nabla \|_{L^\infty} + \lambda^2 \) \| \Psi \|_{L^2 \( \S_\cL \)}^2 \\
		&\leqslant \| \( H_{\nabla, t \Phi} - \lambda \id_{L^2 \( \S_\cL \)} \) \Psi \|_{L^2 \( \S_\cL \)}^2 + 2 \lambda \| \( H_{\nabla, t \Phi} - \lambda \id_{L^2 \( \S_\cL \)} \) \Psi \|_{L^2 \( \S_\cL \)} \| \Psi \|_{L^2 \( \S_\cL \)} \\
		& \quad + \( 2 t \| \nabla \Phi \|_{L^\infty} + 2 \| F_\nabla \|_{L^\infty} + \lambda^2 \) \| \Psi \|_{L^2 \( \S_\cL \)}^2 \\
		&\leqslant \( \| \( H_{\nabla, t \Phi} - \lambda \id_{L^2 \( \S_\cL \)} \) \Psi \|_{L^2 \( \S_\cL \)} + \sqrt{2 t \| \nabla \Phi \|_{L^\infty} + 2 \| F_\nabla \|_{L^\infty} + \lambda^2} \| \Psi \|_{L^2 \( \S_\cL \)} \)^2.
	\end{align}
	Thus, we have
	\begin{equation}
		\| \Psi \|_{\cH_{\nabla, t \Phi}} \leqslant \tfrac{1}{c_\Phi t} \| \( H_{\nabla, t \Phi} - \lambda \id_{L^2 \( \S_\cL \)} \) \Psi \|_{L^2 \( \S_\cL \)} + \tfrac{1}{c_\Phi} \| K \Psi \|_{L^2 \( \S_\cL \)} + \tfrac{\sqrt{2 t \| \nabla \Phi \|_{L^\infty} + 2 \| F_\nabla \|_{L^\infty} + \lambda^2}}{c_\Phi t} \| \Psi \|_{L^2 \( \S_\cL \)}.
	\end{equation}
	Thus, if $|\lambda| \leqslant \cE_0$ and $t > T = O \( \tfrac{1}{c_\Phi^2} \sqrt{\| \nabla \Phi \|_{L^\infty}^2 + c_\Phi^2 \| F_\nabla \|_{L^\infty} + c_\Phi^2 \cE_0^2} \)$, then the last term can be absorbed on the left, and we get
	\begin{equation}
		\| \Psi \|_{\cH_{\nabla, t \Phi}} \leqslant \tfrac{2}{c_\Phi} \( \tfrac{1}{t} \| \( H_{\nabla, t \Phi} - \lambda \id_{L^2 \( \S_\cL \)} \) \Psi \|_{L^2 \( \S_\cL \)} + \| K \Psi \|_{L^2 \( \S_\cL \)} \).
	\end{equation}
	Using \cite{BS18}*{Lemma 4.3.9} with $X = \cH_{\nabla, t \Phi}$, $Y = Z = L^2 \( \S_\cL \)$, and $A = H_{\nabla, t \Phi}$, this implies that $H_{\nabla, t \Phi} - \lambda \id_{L^2 \( \S_\cL \)}$ closed image and finite dimensional kernel.

	Now if $|\lambda| \leqslant \cE_0$ and $t > T$, then $H_{\nabla, t \Phi} - \lambda \id_{L^2 \( \S_\cL \)}$ is self-adjoint and has closed image. Thus, $\lambda$ is either an eigenvalue of $H_{\nabla, t \Phi}$ or the image $H_{\nabla, t \Phi} - \lambda \id_{L^2 \( \S_\cL \)}$ is dense and, thus, all of $L^2 \( \S_\cL \)$, in which case $\lambda$ is not in the spectrum. This proves that $\lambda \in \left[ - \cE_0, \cE_0 \right] \cap \Spec \( H_{\nabla, t \Phi} \)$ contains only eigenvalues.

	Since $\( \nabla, \Phi \)$ is a smooth pair and $H_{\nabla, \Phi}$ is elliptic, each $\Psi \in \ker \( H_{\nabla, t \Phi} - \lambda \id_{L^2 \( \S_\cL \)} \)$ is smooth. Furthermore, using \cref{eq:weitz} we get that $|\nabla^* \nabla \Psi| \leqslant C (\lambda, t, \nabla, \Phi) |\Psi|$ and, thus, $\Psi \in L_2^2 \( \S_\cL \)$. This concludes the proof.
\end{proof}

\smallskip

\begin{corollary}
	The operators $H_{\nabla, \Phi}^\pm : \cH_{\nabla, t \Phi}^\pm \rightarrow L^2 \( \S_\cL^\mp \)$ are Fredholm and
	\begin{equation}
		\index_\rl \( H_{\nabla, \Phi}^\pm \) = - \index_\rl \( H_{\nabla, \Phi}^\mp \). \label{eq:opposite_index}
	\end{equation}
\end{corollary}

\begin{proof}
	By \Cref{lemma:spectrum}, for large $t$, $H_{\nabla, t \Phi}: \cH_{\nabla, t \Phi} \rightarrow L^2 \( \S_\cL \)$ has finite dimensional kernel and closed image. Since as a densely defined map on $L^2 \( \S_\cL \)$, $H_{\nabla, \Phi}$ is self-adjoint, its index is zero. Finally, as $H_{\nabla, \Phi}$ is odd, with respect to the parity grading, we get \cref{eq:opposite_index}.
\end{proof}

\smallskip

For the rest of the paper, let
\begin{align}
	\tau_x (t)	&\eqdef \sqrt[m (x) + 1]{t}, \\
	M			&\eqdef \max \( \left\{ \: m (x) \: \middle| \: x \in \cZ_\Phi \: \right\} \), \\
	\tau (t)	&\eqdef \sqrt[M + 1]{t} = \min \( \left\{ \: \tau_x (t) \: \middle| \: x \in \cZ_\Phi \: \right\} \).
\end{align}
We are now ready to state our main theorem.

\begin{theorem}\label{theorem:general}
	Under Assumptions \textbf{A1} through \textbf{A5}, there are positive numbers, $\cE$ and $T$, both dependent only on $(X, g, \cL, \nabla, \Phi)$, such that for all $t > T$, there are subspaces $\cK_t \subseteq \cH_{\nabla, t \Phi}$ with the following properties:
	\begin{enumerate}
		\item $\cK_t$ and $\cK_t^\perp$ are invariant subspaces of $H_{\nabla, t \Phi}$ and the parity operator $\P$.
		\item The dimension of $\cK_t$ is given by
		\begin{equation}
			\dim_\rl \( \cK_t \) = \sum\limits_{x \in \cZ_\Phi} \left| k (x) \right| < \infty. \label{eq:dim_K_t}
		\end{equation}
		\item For all $\Psi \in \cK_t^\perp$, we have that
		\begin{equation}
			\| H_{\nabla, t \Phi} \Psi \|_{L^2 \( \S_\cL^\mp \)} \geqslant \tfrac{\tau (t)}{\cE} \| \Psi \|_{L^2 \( \S_\cL \)}. \label[ineq]{ineq:high_EVAs}
		\end{equation}
		In particular, $\ker \( H_{\nabla, t \Phi} \) \subseteq \cK_t$.
		\item For all $\Psi \in \cK_t$, we have that
		\begin{equation}
			\| H_{\nabla, t \Phi} \Psi \|_{L^2 \( \S_\cL^\mp \)} \leqslant \cE \| \Psi \|_{L^2 \( \S_\cL \)}. \label[ineq]{ineq:low_EVAs}
		\end{equation}
		In particular, $\ker \( H_{\nabla, t \Phi} \) \subseteq \cK_t$.
	\end{enumerate}
	Furthermore, we have that for all $t \in \rl_+$
	\begin{equation}
		\index_\rl \( H_{\nabla, t \Phi}^\pm \) = \pm \smashoperator{\sum\limits_{x \in \cZ_\Phi}} k (x). \label{eq:index}
	\end{equation}
\end{theorem}

\smallskip

\begin{proof}
	In order to make the rest of the computations simpler, let us conformally flatten out a neighborhood of $\cZ_\Phi$. Since $X$ is 2-dimensional, this can always be done. Since this is a smooth and compact perturbation of the metric, it is enough to prove the claims for this metric, by \Cref{lemma:conformal} this . By Assumption \textbf{A3}, $\cZ_\Phi$ is finite, so the conformal factor can be chosen to be bounded, and as it is independent of $t$, it is enough to prove the claims in this new metric. Without any loss of generality we can also assume that $X$ is flat on the ball of radius two around $\cZ_\Phi$, and it is a tubular neighborhood.

	We construct a finite dimensional subspace, $\cV_t \leqslant \cH_{\nabla, t \Phi}$ that, for $t$ large enough, satisfies (for some $\cE^\prime > 0$, independent of $t$)
	\begin{subequations}
	\begin{align}
		\dim \( \cV_t \)		&= \sum\limits_{x \in \cZ_\Phi} \left| k (x) \right|, \label{eq:V_t_dim} \\
		\forall \Psi \in \cV_t : \quad \| H_{\nabla, t \Phi} \Psi \|_{L^2 \( \S_\cL^\mp \)} &\leqslant \cE^\prime \| \Psi \|_{L^2 \( \S_\cL \)}, \label[ineq]{ineq:upper_bound} \\
		\forall \Psi \in \cV_t^\perp : \quad \| H_{\nabla, t \Phi} \Psi \|_{L^2 \( \S_\cL^\mp \)} &\geqslant \tfrac{\tau (t)}{\cE^\prime} \| \Psi \|_{L^2 \( \S_\cL \)}. \label[ineq]{ineq:lower_bound}
	\end{align}
	\end{subequations}
	Note that such a subspace $\cV_t$ would almost satisfy the requirements of $\cK_t$, expect it is not necessarily invariant with respect to $H_{\nabla, t \Phi}$ and $\P$. In order to construct $\cV_t$, let us first make a few choices. For each $x \in \cZ_\Phi$, choose a smooth function, $\chi_x$, with values in $[0, 1]$, so that $\supp \( \chi_x \) \subseteq B_2 \( x \)$, $\chi_x|_{B_1 \( x \)} \equiv 1$, and $\| \rd \chi_x \|_{L^\infty \( X \)} \leqslant 2$. Fix $x \in \cZ_\Phi$ and assume, without any loss of generality, that $k (x) \geqslant 0$. Let $\nu_x$ and $z$ as in Assumption \textbf{A5}. Then around $x$, $\alpha_x$ is defined via
	\begin{equation}
		\alpha_x \circ z \eqdef \chi_x h \( \nu_x, \cl^+ \( \nabla \nu_x \) \),
	\end{equation}
	and is extended by zero to all of $\cx$, and $f_t \in L_1^2 (\cx)$ be a solution to
	\begin{equation}
		\frac{\del f_t}{\del \overline{z}} + \alpha_x f_t + t \phi_x \overline{f}_t = 0. \label{eq:model_t}
	\end{equation}
	Note that this is an instance of \cref{eq:model_eq}. Let now $\theta_x$ be a flat, unit length trivialization of $\Theta_X$ over the unit disk around $x$, extended smoothly, but otherwise arbitrarily to the rest of $X$, and let $\Psi_{x, f_t} \eqdef \chi_x \( f_t \circ z \) \theta_x \nu_x \in \Gamma \( \S_\cL \)$, extended as zero to the rest of $X$. Since $f_t$ decays exponentially and $\phi_x$ is homogeneous of degree $m (x)$, we can see (through a simple rescaling argument) that there is a $c > 0$, independent of $f_t$, such that for all $w \in \cx$
	\begin{equation}
		|f_t (w)| \leqslant c \tau_x (t) \exp \( - \tfrac{\tau_x (t)}{c} |w| \) \| f_t \|_{L^2 (\cx)}. \label[ineq]{ineq:f_t_bounds}
	\end{equation}
	Repeating the above construction for all $x \in \cZ_\Phi$ and solutions to \cref{eq:model_t}, we get that
	\begin{equation}
		\left\langle \Psi_{x, f_t} \middle| \Psi_{y, f_t^\prime} \right\rangle_{L^2 \( \S_\cL \)} = \delta_{x, y} \langle f_t | f_t^\prime \rangle_{L^2 (\cx)} \( 1 + O \( \tfrac{1}{\tau_x (t)} \) \).
	\end{equation}
	This implies that in this way we construct a linearly independent, in fact, an asymptotically orthonormal, set of vectors. Thus, the resulting spinors span a (real) subspace, $\cV_t \leqslant \cH_{\nabla, t \Phi}$ satisfying \cref{eq:V_t_dim}.

	Next we prove \cref{ineq:upper_bound}. For for any $\Psi \in \cV_t$, of the form $\Psi \eqdef \Psi_{x, f_t}$, using Assumption \textbf{A5}, we have, that
	\begin{equation}
		H_{\nabla, t \Phi} \Psi = \( \del \chi_x \) \( f_t \circ z \) \theta_x \nu_x + \chi_x \( \frac{\del f_t}{\del \overline{z}} + \alpha_x f_t + t \phi_x \overline{f}_t \) \theta_x \nu_x + O \( t |z|^{m (x) + 1} |f_t| \).
	\end{equation}
	The second term vanishes, by \cref{eq:model_t}. Thus, using \cref{ineq:f_t_bounds} and $|\rd \chi_x| \leqslant 2$, we get
	\begin{align}
		\left| H_{\nabla, t \Phi} \Psi \right|^2	&\leqslant |\del \chi_x|^2 |f_t \circ z|^2 + O \( \chi_x^2 t^2 |z|^{2 m (x) + 2} |f_t \circ z|^2 \) + O \( |\del \chi_x| t^2 |z|^{m (x) + 1} |f_t \circ z|^2 \) \\
			&\leqslant \( 4 c^2 + O \( t^2 \) \) \tau_x (t)^2 \exp \( - \tfrac{2 \tau_x (t)}{c} \) \| f_t \|_{L^2 (\cx)}^2 + O \( \chi_x^2 t^2 |z|^{2 m (x) + 2} |f_t \circ z|^2 \).
	\end{align}
	Integrating over $X$ the above yields
	\begin{equation}
		\| H_{\nabla, t \Phi} \Psi \|_{L^2 \( \S_\cL \)} \leqslant O \( \| \Psi \|_{L^2 \( \S_\cL \)} \).
	\end{equation}
	Since $\cV_t$ is a span of finitely many such spinors, we get \cref{ineq:upper_bound}.

	Now we prove \cref{ineq:lower_bound}. Let now $\Psi \in \cV_t^\perp$ and let $\chi \eqdef 1 - \sum_{x \in \cZ_\Phi} \chi_x$. Using \cref{eq:weitz} and Assumptions \textbf{A1} and \textbf{A2}, we get, for $t$ large enough, that
	\begin{align}
		\| H_{\nabla, t \Phi} \( \chi \Psi \) \|_{L^2 \( \S_\cL \)}^2	&= \| \nabla \( \chi \Psi \) \|_{L^2 \( T^* X \otimes \S_\cL \)}^2 + \langle \( \chi \Psi \) | \cl \( F_\nabla \) \( \chi \Psi \) \rangle_{L^2 \( \S_\cL \)} \\
														& \quad + t \langle \( \chi \Psi \) | \cl \( \cA_{\nabla \Phi} \) \( \chi \Psi \) \rangle_{L^2 \( \S_\cL \)} + t^2 \| \cA_\Phi \( \chi \Psi \) \|_{L^2 \( \S_\cL \)}^2 \\
														&\geqslant c^2 t^2 \| \chi \Psi \|_{L^2 \( \S_\cL \)}^2.
	\end{align}

	For each $x \in \cZ_\Phi$, let $f_x^\pm : \cx \rightarrow \cx$ be compactly supported function, such that
	\begin{equation}
		\chi_x \Psi = \( \( f_x^+ \circ z \) \theta_x + \( f_x^- \circ z \) \overline{\theta}_x \) \nu_x.
	\end{equation}
	Then $\Psi \in \cV_t^\perp$ (and that we assumed $k(x) \geqslant 0$) implies that $f_x^\pm$ is $L^2$-perpendicular to the space of solutions to \cref{eq:model_t} (or to the dual equation). Furthermore
	\begin{align}
		\left| \chi_x \Psi \right|^2				&= \left| f_x^+ \circ z \right|^2 + \left| f_x^- \circ z \right|^2, \\
		H_{\nabla, t \Phi} \( \chi_x \Psi \)	&= \( \( D_{\alpha_x}^+ f_x^+ \circ z \) \theta_x + \( D_{\alpha_x}^- f_x^- \circ z \) \overline{\theta}_x \) \nu_x + O \( |z|^{m(x) + 1} \left| \chi_x \Psi \right| \).
	\end{align}
	Thus, using \cref{ineq:gap}, we get
	\begin{align}
		\| H_{\nabla, t \Phi} \( \chi_x \Psi \) \|_{L^2 \( \S_\cL \)}	&\geqslant \| D_{\alpha_x}^+ f_x^+ \|_{L^2 (\cx)} + \| D_{\alpha_x}^- f_x^- \|_{L^2 (\cx)} - O \( \| \chi_x \Psi \|_{L^2 \( \S_\cL \)} \) \\
			&\geqslant O \( \tau_x \( \| f_x^+ \|_{L^2 (\cx)} + \| f_x^- \|_{L^2 (\cx)} \) \) - O \( \| \chi_x \Psi \|_{L^2 \( \S_\cL \)} \) \\
			&= O \( \tau_x (t) \| \chi_x \Psi \|_{L^2 \( \S_\cL \)} \).
	\end{align}
	Since $\chi = 1 - \sum_{x \in \cZ_\Phi} \chi_x$ and if $x \neq y$, then $\chi_x \chi_y = 0$, we get that (for some constants $c_i > 0$)
	\begin{align}
		\| H_{\nabla, t \Phi} \Psi \|_{L^2 \( \S_\cL \)}^2	&= \| H_{\nabla, t \Phi} \( \chi \Psi \) + H_{\nabla, t \Phi} \( \( 1 - \chi \) \Psi \) \|_{L^2 \( \S_\cL \)}^2 \\
			&= \| H_{\nabla, t \Phi} \( \chi \Psi \) \|_{L^2 \( \S_\cL \)}^2 + \sum\limits_{x \in \cZ_\Phi} \| H_{\nabla, t \Phi} \( \chi_x \Psi \) \|_{L^2 \( \S_\cL \)}^2 \\
			& \quad + 2 \left\langle H_{\nabla, t \Phi} \( \chi \Psi \) \middle| H_{\nabla, t \Phi} \( \( 1 - \chi \) \Psi \) \right\rangle_{L^2 \( \S_\cL \)} \\
			&\geqslant c_1 t^2 \| \chi \Psi \|_{L^2 \( \S_\cL \)}^2 + c_2 \sum\limits_{x \in \cZ_\Phi} \( \tau_x (t) \)^2 \| \chi_x \Psi \|_{L^2 \( \S_\cL \)} \\
			& \quad - \| \cl \( \rd \chi \) \Psi \|_{L^2 \( \S_\cL \)}^2 + 2 \int\limits_X \chi (1 - \chi) \left| H_{\nabla, t \Phi} \Psi \right|^2 \vol_g \\
			& \quad + 2 \left\langle \cl \( \rd \chi \) \Psi \middle| (1 - 2 \chi) H_{\nabla, t \Phi} \Psi \right\rangle_{L^2 \( \S_\cL \)} \\
			&\geqslant c_3 \( \tau (t) \)^2 \int\limits_X \( |\chi|^2 + \sum\limits_{x \in \cZ_\Phi} |\chi_x|^2 \) |\Psi|^2 \vol_g - c_4 \| \Psi \|_{L^2 \( \S_\cL \)}^2 \\
			& \quad - c_5 \| \Psi \|_{L^2 \( \S_\cL \)} \| H_{\nabla, t \Phi} \Psi \|_{L^2 \( \S_\cL \)} \\
			&\geqslant c_6 \( \tau (t) \)^2 \| \Psi \|_{L^2 \( \S_\cL \)}^2 - c_5 \| \Psi \|_{L^2 \( \S_\cL \)} \| H_{\nabla, t \Phi} \Psi \|_{L^2 \( \S_\cL \)}.
	\end{align}
	Rearranging the above inequality proves \cref{ineq:lower_bound}.

	We are now ready to define $\cK_t$. By \Cref{lemma:spectrum}, for any $\cE_0 > 0$ and $T > 0$ large enough, the set $\left[ - \cE_0, \cE_0 \right] \cap \Spec \( H_{\nabla, t \Phi} \)$ contains only eigenvalues. Let $T$ be the maximum required by \Cref{lemma:spectrum} for $\cE_0 \eqdef \cE^\prime$ and the definition of $\cV_t$ above. Let $\cK_t$ be the span of the corresponding eigenvectors. By construction, $\cK_t$ is invariant under $H_{\nabla, t \Phi}$ and since $\P$ reverses eigenvalues, $\cK_t$ is also invariant under $\P$. In particular, the set $\left[ - \cE_0, \cE_0 \right] \cap \Spec \( H_{\nabla, t \Phi} \)$ is symmetric, even with respect to multiplicities. Furthermore, \cref{ineq:low_EVAs} is immediately satisfied with $\cE = \cE^\prime$. Note that both \cref{ineq:low_EVAs,ineq:high_EVAs} remain true if $\cE$ is increased.

	Next, we show that $\dim_\rl \( \cK_t \) = \dim_\rl \( \cV_t \) = \sum_{x \in \cK_\Phi} |k (x)|$. By contradiction, if $\dim_\rl \( \cK_t \) > \sum_{x \in \cK_\Phi} |k (x)|$, then $\cK_t \cap \cV_t^\perp$ is nontrivial, but that would contradict \cref{ineq:lower_bound}. Similarly, if $\dim_\rl \( \cK_t \) < \sum_{x \in \cK_\Phi} |k (x)|$, then $\cK_t^\perp \cap \cV_t$ is nontrivial, but that would contradict \cref{ineq:upper_bound}. This proves \cref{eq:dim_K_t}.

	In order to show \cref{ineq:high_EVAs}, let $\Pi_t$ be the orthogonal projection from $L^2 \( \S_\cL \)$ onto $\cV_t$. If $\Psi \in \cK_t$, then, using \cref{ineq:upper_bound,ineq:lower_bound}
	\begin{align}
		\cE^\prime \| \Psi \|_{L^2 \( \S_\cL \)}	&\geqslant \| H_{\nabla, t \Phi} \Psi \|_{L^2 \( \S_\cL \)} \\
			&= \| H_{\nabla, t \Phi} \( \( \id_{L^2 \( \S_\cL \)} - \Pi_t \) + \Pi_t \) \Psi \|_{L^2 \( \S_\cL \)} \\
			&\geqslant \| H_{\nabla, t \Phi} \( \id_{L^2 \( \S_\cL \)} - \Pi_t \) \Psi \|_{L^2 \( \S_\cL \)} - \| H_{\nabla, t \Phi} \Pi_t \Psi \|_{L^2 \( \S_\cL \)} \\
			&\geqslant \tfrac{\tau (t)}{\cE^\prime} \| \( \id_{L^2 \( \S_\cL \)} - \Pi_t \) \Psi \|_{L^2 \( \S_\cL \)} - \cE^\prime \| \Pi_t \Psi \|_{L^2 \( \S_\cL \)} \\
			&\geqslant \tfrac{\tau (t)}{\cE^\prime} \| \( \id_{L^2 \( \S_\cL \)} - \Pi_t \) \Psi \|_{L^2 \( \S_\cL \)} - \cE^\prime \| \Psi \|_{L^2 \( \S_\cL \)},
	\end{align}
	thus
	\begin{equation}
		\| \( \id_{L^2 \( \S_\cL \)} - \Pi_t \) \Psi \|_{L^2 \( \S_\cL \)} \leqslant \tfrac{\( \cE^\prime \)^3}{\tau (t)} \| \Psi \|_{L^2 \( \S_\cL \)}. \label[ineq]{ineq:V_perp_bound}
	\end{equation}
	This implies that (using $t > T$ and potentially increasing $T$ once more)
	\begin{equation}
		\| \Pi_t \Psi \|_{L^2 \( \S_\cL \)}	\geqslant \( 1 - O \( \tfrac{1}{\sqrt{\tau (t)}} \) \) \| \Psi \|_{L^2 \( \S_\cL \)}.
	\end{equation}
	Thus, $\Pi_t$ induces a linear isomorphism that is almost isometric.

	Now let $\Psi \in \cK_t^\perp$ and let pick the unique $\Psi_0 \in \cK_t$, such that $\Pi_t \Psi = \Pi_t \Psi_0$. Then, using \cref{ineq:V_perp_bound}
	\begin{align}
		\| \Pi_t \Psi \|_{L^2 \( \S_\cL \)}^2	&= \left\langle \Psi \middle| \Pi_t \Psi \right\rangle_{L^2 \( \S_\cL \)} \\
			&= \left\langle \Psi \middle| \Pi_t \Psi_0 \right\rangle_{L^2 \( \S_\cL \)} \\
			&= - \left\langle \Psi \middle| \( \id_{L^2 \( \S_\cL \)} - \Pi_t \) \Psi_0 \right\rangle_{L^2 \( \S_\cL \)} \\
			&\leqslant \| \Psi \|_{L^2 \( \S_\cL \)} \| \( \id_{L^2 \( \S_\cL \)} - \Pi_t \) \Psi_0 \|_{L^2 \( \S_\cL \)} \\
			&\leqslant \| \Psi \|_{L^2 \( \S_\cL \)} \tfrac{\( \cE^\prime \)^3}{\tau (t)} \| \Psi_0 \|_{L^2 \( \S_\cL \)} \\
			&\leqslant \| \Psi \|_{L^2 \( \S_\cL \)} \tfrac{O (1)}{\tau (t)} \| \Pi_t \Psi_0 \|_{L^2 \( \S_\cL \)} \\
			&= \| \Psi \|_{L^2 \( \S_\cL \)} \tfrac{O (1)}{\tau (t)} \| \Pi_t \Psi \|_{L^2 \( \S_\cL \)},
	\end{align}
	and thus $\| \Pi_t \Psi \|_{L^2 \( \S_\cL \)} \leqslant \tfrac{O (1)}{\tau (t)} \| \Psi \|_{L^2 \( \S_\cL \)}$, which yields
	\begin{align}
		\| H_{\nabla, t \Phi} \Psi \|_{L^2 \( \S_\cL \)}	&= \| H_{\nabla, t \Phi} \( \( \id_{L^2 \( \S_\cL \)} - \Pi_t \) + \Pi_t \) \Psi \|_{L^2 \( \S_\cL \)} \\
			&\geqslant \| H_{\nabla, t \Phi} \( \id_{L^2 \( \S_\cL \)} - \Pi_t \) \Psi \|_{L^2 \( \S_\cL \)} - \| H_{\nabla, t \Phi} \Pi_t \Psi \|_{L^2 \( \S_\cL \)} \\
			&\geqslant \tfrac{\tau (t)}{\cE^\prime} \| \( \id_{L^2 \( \S_\cL \)} - \Pi_t \) \Psi \|_{L^2 \( \S_\cL \)} - \cE^\prime \| \Pi_t \Psi \|_{L^2 \( \S_\cL \)} \\
			&\geqslant \tfrac{\tau (t)}{\cE^\prime} \| \Psi \|_{L^2 \( \S_\cL \)} - \( \tfrac{\tau (t)}{\cE^\prime} + \cE^\prime \) \| \Pi_t \Psi \|_{L^2 \( \S_\cL \)} \\
			&\geqslant \( \tfrac{\tau (t)}{\cE^\prime} - O(1) \) \| \Psi \|_{L^2 \( \S_\cL \)},
	\end{align}
	which, for some $\cE > 0$, yields \cref{ineq:high_EVAs}.

	Now, let us prove \cref{eq:index}. Let $\cV_t^\pm$ be the (anti-)chiral part of $\cV_t$ and $\Pi_t^\pm$ be the $L^2$-orthogonal projections from $\cH_{\nabla, t \Phi}^\pm$ onto $\cV_t^\pm$. The operators
	\begin{equation}
		\widetilde{H}^\pm \eqdef \( \id_{\cH_{\nabla, t \Phi}^\mp} - \Pi_t^\mp \) \circ H_{\nabla, t \Phi}^\pm \circ \( \id_{\cH_{\nabla, t \Phi}^\pm} - \Pi_t^\pm \) : \cH_{\nabla, t \Phi}^\pm \rightarrow L^2 \( \S_\cL^\mp \).
	\end{equation}
	are compact (in fact, finite rank) perturbations of $H_{\nabla, t \Phi}^\pm$ and, thus, are also Fredholm and have the same index. By construction, $\ker \( \widetilde{H}^\pm \) = \cV_t^\pm$, and hence
	\begin{equation}
		\dim_\rl \( \ker \( \widetilde{H}^\pm \) \) = \dim_\rl \( \cV_t^\pm \) = \smashoperator{\sum\limits_{\substack{x \in \cZ_\Phi \\ \sign \( k (x) \) = \pm}}} \pm k (x),
	\end{equation}
	and similarly, it is easy to see that, $\coker \( \widetilde{H}^\pm \) = \cV_t^\mp$, and hence
	\begin{equation}
		\dim_\rl \( \coker \( \widetilde{H}^\pm \) \) = \dim_\rl \( \cV_t^\mp \) = \smashoperator{\sum\limits_{\substack{x \in \cZ_\Phi \\ \sign \( k (x) \) = \mp}}} \mp k (x).
	\end{equation}
	Thus, the index satisfies
	\begin{equation}
		\index_\rl \( H_{\nabla, t \Phi}^\pm \) = \index_\rl \( \widetilde{H} \) = \smashoperator{\sum\limits_{\substack{x \in \cZ_\Phi \\ \sign \( k (x) \) = \pm}}} \pm k (x) - \smashoperator{\sum\limits_{\substack{x \in \cZ_\Phi \\ \sign \( k (x) \) = \mp}}} \mp k (x) = \pm \smashoperator{\sum\limits_{x \in \cZ_\Phi}} k (x),
	\end{equation}
	which concludes the proof of \cref{eq:index}.
\end{proof}

\bigskip

\section{Applications}
\label{sec:applications}

\subsection{The Majorana bundle over the vortex moduli space}

Once a Hermitian Clifford module bundle over a (pseudo-)Riemannian manifold is fixed, the only remaining parameters of generalized Jackiw--Rossi operators are a compatible connection $\nabla$ is a section $\Phi$. In certain situations there are (mathematically and physically) relevant such families that form a smooth moduli.

One such example is Ginzburg--Landau theory, more concretely, $\tau$-vortices. Let us briefly summarize the theory of $\tau$-vortices in dimension two. Let $\( \Sigma, g \)$ be a closed, oriented Riemannian surface and $\cV \rightarrow \Sigma$ be a Hermitian line bundle of positive degree $d$. Let $\omega$ be the K\"ahler/volume form and $\Lambda : \bigwedge^2 \rightarrow \bigwedge^0$ be the contraction by $\omega$. For each smooth unitary connection $\nabla$ and smooth section $\Phi$ on $\cV$ the (critically coupled) Ginzburg--Landau energy is
\begin{equation}
	E_\tau \( \nabla, \Phi \) = \int\limits_\Sigma \( \left| F_\nabla \right|^2 + \left| \nabla \Phi \right|^2 + \tfrac{1}{4} \( \tau - |\Phi|^2 \)^2 \) \: \vol. \label{eq:gle}
\end{equation}
One can perform a ``Bogomolny trick'' and rewrite the energy \eqref{eq:gle} as
\begin{equation}
	E_\tau \( \nabla, \Phi \) = \int\limits_\Sigma \( \left| i \Lambda F_\nabla - \tfrac{1}{2} \( \tau - |\Phi|^2 \) \right|^2 + 2 |\delbar_\nabla \Phi|^2 \) \vol + 2 \pi \tau d. \label{eq:Bogomolny_trick}
\end{equation}
Thus, absolute minimizers of the energy \eqref{eq:gle} are characterized by
\begin{subequations}
	\begin{align}
		i \Lambda F_\nabla	&= \tfrac{1}{2} \( \tau - |\Phi|^2 \), \label{eq:vortex_1} \\
		\delbar_\nabla \Phi	&= 0. \label{eq:vortex_2}
	\end{align}
\end{subequations}
We call \cref{eq:vortex_1,eq:vortex_2} the \emph{$\tau$-vortex equations} and solutions $\( \nabla, \Phi \)$ \emph{$\tau$-vortex fields}.

Let $\tau_0 \eqdef \tfrac{4 \pi d}{\Ar (\Sigma, g)}$. In \cite{bradlow_vortices_1990}*{Theorem~4.6} Bradlow showed that when $\tau < \tau_0$, then there are no $\tau$-vortex fields, when $\tau = \tau_0$, then $\( \nabla, \Phi \)$ is a $\tau$-vortex field exactly when $\nabla$ is a $\rU (1)$ Yang--Mills field and $\Phi = 0$, and when $\tau > \tau_0$ then all $\tau$-vortex fields are irreducible, that is $\Phi \neq 0$, and the moduli space of gauge equivalence classes of $\tau$-vortex fields is canonically isomorphic to $\Sym^d \( \Sigma \)$. Note that $\Phi$ is holomorphic, and let $\cD_\Phi \in \Sym^d \( \Sigma \)$ be its (effective) divisor. The canonical isomorphism is then given by $\left[ \nabla, \Phi \right] \mapsto \cD_\Phi$. Here gauge equivalence refers to the action of $\gamma \in \cG \eqdef L_2^2 \( \Sigma; \rU (1) \)$ via
\begin{equation}
	\gamma \( \nabla, \Phi \) = \( \gamma \circ \nabla \circ \gamma^{- 1}, \gamma \Phi \).
\end{equation}
Using the metric $g$ and a reference point $x_0 \in \Sigma$, we can canonically factor $\cG$ as follows: Let us identify $H^1 \( \Sigma; \Z \)$ with the group of harmonic functions from $\Sigma$ to $\rU (1)$ that are one at $x_0$ and $L_{2, x_0}^2 \( \Sigma; \rl \)$ be the space of $L_1^2$ function that vanish at $x_0$. Then the following map is an isomorphism
\begin{equation}
	\rU (1) \oplus H^1 \( \Sigma; \Z \) \oplus L_{2, x_0}^2 \( \Sigma; i \rl \) \rightarrow \cG; \quad (\lambda, \gamma_H, i f) \mapsto \lambda \gamma_H \exp \( i f \). \label{eq:factorization}
\end{equation}
Let us fix $\tau > \tau_0$ and let $\cN$ be space of all $\tau$-vortex fields determined by these data, that is
\begin{equation}
	\cN \eqdef \left\{ \: \( \nabla, \Phi \) \: \middle| \: \( \nabla, \Phi \) \mbox{ solves \cref{eq:vortex_1,eq:vortex_2}} \: \right\}.
\end{equation}
Then $\cN$ is a Hilbert manifold modeled on the $L_1^2$-completion of $i \Omega^1 \oplus \Gamma \( \cL^{- 2} \)$, and the factor $\cN / L_{2, x_0}^2 \( \Sigma; i \rl \)$ is a finite dimensional manifold of (real) dimension $2 d + 1$. The actions of $\rU (1)$ and $H^1 \( \Sigma; \Z \)$ are also easy to understand: Note that $H^1 \( \Sigma; \Z \) \cong \pi_1 \( \Sym^d \( \Sigma \) \)$ and so $\cN / \( H^1 \( \Sigma; \Z \) \oplus L_{2, x_0}^2 \( \Sigma; i \rl \) \)$ is a principal $\rU (1)$ bundle over $\Sym^d \( \Sigma \)$.

Let now $d = 2 k > 0$ be positive and even, write $\cV = \cL^{- 2}$ (thus, the degree of $\cL$ is $- k < 0$), and let $\S_\cL$ be as in \Cref{sec:JR_on_manifolds}. We can then associate a Jackiw--Rossi operator to each $\tau$-vortex field, $\( \nabla, \Phi \)$. Let us prove a technical lemma.

\begin{lemma}
\label{lemma:cross_terms}
	For any $\nabla$ and $\Phi$ we have
	\begin{align}
		\cl^+ \( \cA_{\nabla \Phi} \) &= \cl^+ \( \cA_{\delbar_\nabla \Phi} \), \\
		\cl^- \( \cA_{\nabla \Phi} \) &= \cl^- \( \cA_{\del_\nabla \Phi} \).
	\end{align}
	Thus, if $\delbar_\nabla \Phi = 0$, then for all $\Psi^+ \in \cH_{\nabla, \Phi}^+$
	\begin{equation}
		\| H_{\nabla, \Phi}^+ \Psi^+ \|_{L^2 (\S_\cL^-)}^2 = \| \D_\nabla^+ \Psi^+ \|_{L^2 (\S_\cL^-)}^2 + \| \cA_\Phi^+ \Psi^+ \|_{L^2 (\S_\cL^-)}^2,
	\end{equation}
	and, in particular, if $\Phi$ is not identically zero, then $\ker \( H_{\nabla, \Phi}^+ \)$ is trivial and $\ker \( H_{\nabla, \Phi}^- \)$ has real dimension $d$.
\end{lemma}

\begin{proof}
	This is a local computation.
\end{proof}

Now the assignment $\( \nabla, \Phi \) \mapsto \ker \( H_{\nabla, \Phi}^- \)$ defines a smooth, real, rank $k$ vector bundle over $\cN$, which we call $\widetilde{\cK}$. In order to descend this bundle to the vortex moduli space, we have to understand how gauge transformations act on $\widetilde{\cK}$.

Assume that $\gamma = \mu^2$ for some $\mu \in \cG$. Then $\Psi \in \ker \( H_{\nabla, \Phi}^- \)$ exactly if $\mu \Psi \in \ker \( H_{\gamma \( \nabla, \Phi \)}^- \)$. Thus, to descend the kernel bundle $\widetilde{\cK}$ to the vortex modulo space, one needs to take square roots of gauge transformations, but that cannot always be done for two independent reasons:
\begin{enumerate}
	\item Let us define a loop in $\rU (1) \subseteq \cG$ via $t \mapsto \gamma_t \eqdef \exp \( 2 i t \)$ as $t$ runs from zero to $2 \pi$. Then $t \mapsto \mu_t \eqdef \exp \( i t \)$ is also a loop in $\cG$ and satisfies $\mu_t^2 = \gamma_t$, but $\mu_\pi = - 1 \neq 1$. Thus, nonzero elements of $\widetilde{\cK}$ would be identified with their opposites. This issue can be remedied if we only attempt to descend the projectivization of $\widetilde{\cK}$.

	\item If $\gamma_H$ is a harmonic, $\rU (1)$-valued function that represents an integer cohomology class that is not the double of another class, then $\gamma_H$ is a gauge transformation without a square root.
\end{enumerate}
As $L_{2, x_0}^2 \( \Sigma; i \rl \)$ is a vector space, taking square roots of gauge transformations coming from $L_{2, x_0}^2 \( \Sigma; i \rl \)$ can always be done.

\smallskip

Note that $2 \: H^1 \( \Sigma; \Z \) = \( \Sigma; 2 \Z \)$, and let $\cG^\prime \eqdef \cG^2 \cong \rU (1) \oplus H^1 \( \Sigma; 2 \Z \) \oplus L_{2, x_0}^2 \( \Sigma; i \rl \)$. The above arguments prove the following:

\begin{theorem}
	Let $\widehat{\cM} \eqdef \cN / \( H^1 \( \Sigma; 2 \Z \) \oplus L_{2, x_0}^2 \( \Sigma; i \rl \) \)$. Then $\widetilde{\cM}$ is a smooth, closed manifold of (real) dimension $4 k + 1$ and the vector bundle $\widetilde{\cK}$ descends smoothly to $\widehat{\cM}$, call $\cK$. Furthermore, $\rU (1)$ has a smooth, free action on $\widehat{\cM}$, which locally ascends to $\cK$, with monodromy $- 1$.

	Let $\cM \eqdef \cN / \cG^\prime \cong \widehat{\cM} / \rU (1)$. Then $\cM$ is a principal $H^1 \( \Sigma; \Z_2 \) \cong \Z_2^{b_1 \( \Sigma \)}$ bundle over $\Sym^{2k} \( \Sigma \)$, corresponding to the $H^1 \( \Sigma; 2 \Z \) \subset \pi_1 \( \Sym^{2k} \( \Sigma \) \)$. Furthermore, the projectivization of $\cK$, call $\P \( \cK \)$, descends to $\cM$.
\end{theorem}

As a concluding thought, we look at the simplest example, that is when $\Sigma$ is the 2-sphere. Let $\Sigma = S^2$ with the round metric or radius one, let $\cL \cong \cO (- 1)$. Then $H^1 \( \Sigma; \Z \)$ is trivial and for any positive integer $d$, we have $\Sym^d \( S^2 \) \cong \cx \P^n$. The Coulomb bundle over $\Sym^d \( S^2 \)$ is the Hopf fibration $S^{2 d + 1} \rightarrow \cx \P^d$. Thus, as $b_1 \( S^2 \) = 0$, if $d = 2 k$, then $\widehat{\cM} = S^{4 k + 1}$ and $\cM \eqdef \Sym^{2 k} \( S^2 \) \cong \cx \P^{2 k}$. Furthermore, $\cK$ is a rank $2 k$ real vector bundle over $S^{4 k + 1}$. Such bundles are necessarily orientable and, thus, are classified by homotopy types of maps from $S^{4 k}$ to $\SO (2 k)$, that is, by $\pi_{4 k} \( \SO (2 k) \)$. When $k = 1$, this group is $\pi_4 \( \SO (2) \) = \pi_4 \( S^1 \)$ which is trivial. Thus, in this case $\cK$ is a trivial bundle. When $k = 2$, then $\pi_8 \( \SO (4) \) \cong \Z_2^2$ is the Klein group and, thus, we cannot conclude the triviality of $\cK$.

\medskip

\subsection{Solutions on spacetimes}

In this section we use \Cref{theorem:general}, to construct solutions in higher dimensions.

Let $\Sigma$ be a oriented, complete Riemannian surface, and let $\S_\cL$ be a spinor bundle corresponding to a spin$^{\mathrm{c}}$ structure of $X$ with associated line bundle $\cL^2$. Let $\nabla$ be a compatible connection on $\S_\cL$ and $\Phi$ be a smooth section of $\cF^+ = \cL^{- 2}$, both satisfying the conditions \textbf{A1} through \textbf{A5} in \Cref{sec:JR_on_surfaces}. Thus, we get a generalized Jackiw--Rossi operator, $H_{\nabla, \Phi}$.

Let $Y$ be a smooth, oriented, Lorentzian surface, that is a pseudo-Riemannian manifold of signature $(1, 1)$. Assume that $\E_Y \rightarrow Y$ is a Hermitian Clifford module bundle with compatible connection $\nabla^Y$, and a flat, compatible, real structure, call $\cA_Y$. Let $\ker \( \D_Y \)$ be the vector space of smooth solutions to the Dirac equation on $\E_Y$. We impose no integrability requirements this time. Note that $\cA_Y$ induces a real structures on $\ker \( \D_Y^\pm \)$ and if $\Psi^{Y, \pm} \in \ker \( \D_Y^\pm \)$, then $\tfrac{1}{2} \( \Psi^{Y, \pm} + \cA_Y^\pm \Psi^{Y, \pm} \) \in \ker \( \D_Y^\pm \)$ is a real vector with respect to $\cA_Y^\pm$.

\begin{remark}
	Examples of $\E_Y \rightarrow Y$ include the spinor bundle of $Y$ when it exists. Furthermore, the spinor bundle can be twisted by flat, real bundles.
\end{remark}

Let $X = Y \times \Sigma$, which is a Lorentzian 4-manifold, that is a \emph{spacetime}. Following \cite{ABS64}*{\textsection 6}, let
\begin{equation}
	\E_X \eqdef \E_X^+ \oplus \E_X^-, \quad \mbox{with} \quad \E_X^\pm \eqdef \( \pi_Y^* (\E_Y^+) \otimes \pi_\Sigma^* (\E_\Sigma^\pm) \) \oplus \( \pi_Y^* (\E_Y^\mp) \otimes \pi_\Sigma^* (\E_\Sigma^-) \). \label{eq:E_def}
\end{equation}
Now $\E_X$ is a Clifford module bundle over $X$ (via \cite{ABS64}*{Equation~(6.1)}, equipped with a compatible connection $\nabla^X$ induced by $\nabla$ and $\nabla^Y$, and a conjugate linear map $\cA_X \eqdef \cA_Y \otimes \cA_\Phi$. It is easy to see that $\cA_X$ has the form $\cA_{\Phi^X}$ from \cref{eq:A_Phi} with $\Phi^X$ being the pullback of $\Phi$ to $X$. In particular, the zero locus of $\Phi^X$ is $Y \times \cZ_\Phi$, which is a collection of embedded Lorentzian minimal surfaces, that is \emph{worldsheets}.

The above data gives us a generalized Jackiw--Rossi operator on $\E_X \rightarrow X$:
\begin{equation}
	H_{\nabla^X, \Phi^X} = \D_{\nabla^X} + \cA_{\Phi^X}.
\end{equation}

We are now ready to state the theorem of this section.

\begin{theorem}
	Let
	\begin{equation}
		\Psi^Y \eqdef \begin{pmatrix} \Psi^{Y, +} \\ \Psi^{Y, -} \end{pmatrix} \in \ker \( \D_{\nabla^Y}^+ \) \oplus \ker \( \D_{\nabla^Y}^- \),
	\end{equation}
	and assume that $\cA_Y \Psi^Y = \Psi^Y$. Then for each $\lambda \in \Spec \( H_{\nabla, \Phi} \)$ and for each
	\begin{equation}
		\Psi^\Sigma = \begin{pmatrix} \Psi^{\Sigma, +} \\ \Psi^{\Sigma, -} \end{pmatrix} \in \ker \( H_{\nabla, \Phi} - \lambda \id \),
	\end{equation}
	the spinor
	\begin{equation}
		\Psi^X \eqdef \begin{pmatrix} \Psi^{Y, +} \otimes \Psi^{\Sigma, +} + \Psi^{Y, -} \otimes \Psi^{\Sigma, -} \\ \Psi^{Y, +} \otimes \Psi^{\Sigma, -} + \Psi^{Y, -} \otimes \Psi^{\Sigma, +} \end{pmatrix},
	\end{equation}
	is in the kernel of $H_{\nabla^X, \Phi^X} - \lambda \id$. In particular, when $\Phi$ is replaced by $t \Phi$, then these spinors concentrate around $Y \oplus \cZ_\Phi$ as $t \rightarrow \infty$, according to \Cref{theorem:general}.
\end{theorem}

\begin{remark}
	Similar constructions can be performed when $Y$ is a pseudo-Riemannian manifold of any type $(t, s)$, equipped with a Hermitian Clifford module bundle with compatible connection, and a flat, compatible, real structure. We focused on the case of a Lorentzian manifold for the obvious mathematical physical importance.
\end{remark}

\bigskip

\appendix

\section{Bilinear pairings and conjugate linear maps}
\label{sec:pairings}

\subsection{Pairings of spinor modules}
\label{sec:spinor_modules}

Let $\S = \S_{t, s}$ be the spinor representation of the $d \eqdef (t + s)$-dimensional (real) Clifford algebra, $\CL_{t, s} = \CL_{t, s}^0 \oplus \CL_{t, s}^1$. Here $t$ is the number of timelike (negative definite) and $s$ is the number of spacelike (positive definite) directions, and let $r \eqdef t - s$. Denote the Hermitian structure and the Clifford multiplication on $\S$, by $h$ and $\cl$, respectively. In even dimensions the spinor representation is graded, $\S = \S^+ \oplus \S^-$. The grading is also called the \textit{parity}, and the grading operator is denoted by $\P$. An operator acting on spinors is even, if it preserves parity (commutes with $\P$), and \textit{odd}, if it interchanges it (anti-commutes with $\P$).

In every dimension there are (essentially unique) nondegenerate, $\CL_{t, s}^0$-invariant, metric compatible, bilinear forms on $\S$, which we call $\cB$. A concise way to think about these pairing is given via the (essentially canonical) isomorphisms of representations:
\begin{equation}\label{eq:S_square}
	\S \otimes \S \cong \left\{
	\begin{array}{ll}
		\bigwedge_\cx^{\mathrm{even}} \( \rl^{t, s} \),	& \mbox{if $t + s$ is odd,} \\
		\bigwedge_\cx^* \( \rl^{t, s} \),				& \mbox{if $t + s$ is even}.
	\end{array} \right.
\end{equation}
Denote the metric dual of $\cB$ by $\cA$, which is then an anti-unitary map of $\S$. The $\cB$ can be recovered from $\cA$ via
\begin{equation}
	\cB \( \Psi_1, \Psi_2 \) = h (\cA \Psi_1, \Psi_2).
\end{equation}
The choice of $\cA$ is always ambiguous up to the action of $\rU (1)$ via $\cA \mapsto \lambda \cA$. Note that $\( \lambda \cA \)^2 = \cA^2$. In odd dimensions these maps and, thus, the forms, are unique, up to the action of $\rU (1)$, but in even dimensions there is another sign ambiguity, which we explain below. Let $\equiv_n$ be the modulo $n$ equivalence relation of integers and let us denote the restrictions of $\cA$ to $\S^\pm$ by $\cA^\pm$, in even dimensions. Following \cite{V04}*{Theorem~6.5.7}, we have the following:
\begin{itemize}[leftmargin=0.6in]

	\item[$r \equiv_8 0$:] $\cA$ is even and $\cA^\pm \circ \cl^\mp = \sigma \cl^\mp \circ \cA^\mp$, where $\sigma = \pm 1$ is an extra degree of freedom in the choice of $\cA$, and $\cA^2 = \id_\S$.

	\item[$r \equiv_8 1$:] $\cA$ anti-commutes with $\cl$ and $\cA^2 = \id_\S$.

	\item[$r \equiv_8 2$:] $\cA$ is odd, $\cA^\pm \circ \cl^\mp = \sigma \cl^\pm \circ \cA^\mp$, and $\cA^\mp \circ \cA^\pm = - \sigma \id_{\S^\pm}$, where $\sigma = \pm 1$ is an extra degree of freedom in the choice of $\cA$.

	\item[$r \equiv_8 3$:] $\cA$ commutes with $\cl$ and $\cA^2 = - \id_\S$.

	\item[$r \equiv_8 4$:] $\cA$ is even and $\cA^\pm \circ \cl^\mp = \sigma \cl^\mp \circ \cA^\mp$, where $\sigma = \pm 1$ is an extra degree of freedom in the choice of $\cA$, and $\cA^2 = - \id_\S$.

	\item[$r \equiv_8 5$:] $\cA$ anti-commutes with $\cl$ and $\cA^2 = \id_\S$.

	\item[$r \equiv_8 6$:] $\cA$ is odd, $\cA^\pm \circ \cl^\mp = \sigma \cl^\pm \circ \cA^\mp$, and $\cA^\mp \circ \cA^\pm = \sigma \id_{\S^\pm}$, where $\sigma = \pm 1$ is an extra degree of freedom in the choice of $\cA$.

	\item[$r \equiv_8 7$:] $\cA$ commutes with $\cl$ and $\cA^2 = \id_\S$.

\end{itemize}

Recall that we defined $\sigma_r$ and $\s_r$ in \cref{eq:sigma_s_def}. We now then have

\begin{lemma}
	\label{lemma:pairings}
	The above bilinear form can be extended to forms as follows:
	\begin{itemize}

		\item If $t + s$ is odd, then the map
			\begin{align}
				\widehat{\cB}: \S \otimes \S	&\rightarrow	{\bigwedge}_\cx^{\mathrm{even}} \( \rl^{t, s} \); \\
				\Psi_1 \otimes \Psi_2			&\mapsto		\sum\limits_{\mathclap{\substack{0 \leqslant |I| \leqslant d \\ {|I| \textnormal{ is even}}}}} \cB \( \cl \( \rd x^I \) \Psi_1, \Psi_2 \) \rd x^I,
			\end{align}
			is an isomorphism. Analogous isomorphism exists for odd forms.

		\item If $t + s$ is even, then the map
			\begin{align}
				\widehat{\cB} : \S \otimes \S	&\rightarrow	{\bigwedge}_\cx^* \( \rl^{t, s} \); \\
				\Psi_1 \otimes \Psi_2 			&\mapsto		\sum\limits_{\mathclap{0 \leqslant |I| \leqslant d}} \cB \( \cl \( \rd x^I \) \Psi_1, \Psi_2 \) \rd x^I,
			\end{align}
			is an isomorphism.
	\end{itemize}
	Furthermore, if $\pi_0 : \bigwedge_\cx^{\mathrm{even}} \( \rl^{t, s} \) \rightarrow \cx$ is the orthogonal projection onto $\bigwedge_\cx^0 \( \rl^{t, s} \) \cong \cx$, then $\cB = \pi_0 \circ \widehat{\cB}$.

\end{lemma}

\begin{proof}
	The above maps are linear injections between finite dimensional vector spaces, so one can prove that they are isomorphisms by simple dimension counts.
\end{proof}

\begin{remark}
	Let $\mathrm{vol}$ be the volume form of $\rl^{t, s}$, that is a unit norm element of $\bigwedge^{t + s} \( \rl^{t, s} \)$. Then $\cA^\prime = \cl \( \mathrm{vol} \) \circ \cA$ is also a $\CL_{t, s}^0$-invariant conjugate linear map. In odd dimensions $\cl \( \mathrm{vol} \)$ is proportional to the identity, but in even dimensions $\cl \( \mathrm{vol} \)$ is proportional to the parity operator, and if $\cA \circ \cl (\cdot) = \sigma \cl (\cdot) \circ \cA$, then $\cA^\prime \circ \cl (\cdot) = - \sigma \cl (\cdot) \circ \cA^\prime$.
\end{remark}

\medskip

\subsection{More general Clifford modules}
\label{sec:CL_modules}

Let now $\E$ be any Hermitian $\CL_{t, s}$ module. There always is a Hermitian vector space $E$, such that $\E \cong \S \otimes E$; cf. \cite{N07}*{Corollary~11.1.22.}, in fact $E \cong \Hom_{\CL_d} \( \S, \E \)$. In particular, if $d$ is even, then $\E$ is also $\Z_2$-graded, which we still denote as $\E \cong \E^+ \oplus \E^-$ with $\E^\pm \cong \S^\pm \otimes E$.

Let $F \eqdef E \otimes E$. Then we have the following generalization of the pairing from \Cref{sec:pairings}:
\begin{equation}
	\cB_\E : \E \otimes \E \rightarrow F; \: \( \psi_1 \otimes e_1 \) \otimes \( \psi_2 \otimes e_2 \) \mapsto \cB \( \psi_1, \psi_2 \) e_1 \otimes e_2,
\end{equation}
and more generally we define maps $\widehat{\cB}_\E$, with values in $\bigwedge_\cx^* \( \rl^{t, s} \) \otimes F$, similarly to \Cref{lemma:pairings}. Moreover, we have the following simple lemma.

\begin{lemma}
	\label{lemma:uniqueness}
	For any complex bilinear form, $B$, on $\E$, there exists a $\Phi$ in $\( \bigwedge_\cx^* \( \rl^{t, s} \) \otimes F \)^*$, such that
	\begin{equation}
		B = \Phi \circ \widehat{\cB}_\E. \label{eq:B_decomp}
	\end{equation}
	If $r = t - s$ is even, then $\Phi$ is unique. If $r = t - s$ is odd, then $\Phi$ is unique among even forms (or odd forms for the analogous isomorphism).

	Furthermore, if $\Phi$ is in $\( \bigwedge_\cx^k \( \rl^{t, s} \) \otimes F \)^*$, (where $k$ is even, if $t + s$ is odd), then
	\begin{equation}
		B \( \Psi_1, \Psi_2 \) = (- 1)^{\frac{k (k + 1)}{2}} \s_r \sigma_r^k B \( \Psi_2, \Psi_1 \). \label{eq:B_symmetry}
	\end{equation}
\end{lemma}

\begin{proof}
	The proof involves straightforward computations and we only present the key steps here.

	In all cases, the correspondence
	\begin{equation}
		\Phi \mapsto \Phi \circ \widehat{\cB}_\E, \label{eq:correspondence}
	\end{equation}
	is a linear injection between finite dimensional vector spaces, so one can prove that they are isomorphisms by simple dimension counts. This proves \cref{eq:B_decomp} and the uniqueness of $\Phi$.

	Finally, proving \cref{eq:B_symmetry} is then straightforward, using \cref{eq:sigma_s_def} and that the values of $\cl$ on 1-forms are skew-adjoint with respect to the Hermitian structure of $\E$.
\end{proof}

\smallskip

Now, given an element $\Phi \in \( \bigwedge_\cx^* \( \rl^{t, s} \) \otimes F \)^*$, we can now define a conjugate linear map $\cA_\Phi : \E \rightarrow \E$ via
\begin{equation}
	h^\E \( \cA_\Phi \Psi_1, \Psi_2 \) = \Phi \( \widehat{\cB}_\E \( \Psi_1, \Psi_2 \) \).
\end{equation}
Now let us consider the splitting $\E \otimes \E \cong \vee^2 \E \oplus \wedge^2 \E$. These subspaces are the $\pm 1$ eigenspaces of the involution $\Psi_1 \otimes \Psi_2 \mapsto \Psi_2 \otimes \Psi_1$. Similarly, we get a splitting $F = F^+ \oplus F^-$. Assume now that $\Phi^\pm \in F^\pm$. Then simple computation shows that
\begin{equation}
	h^\E \( \cA_{\Phi^\pm} \Psi_1, \Psi_2 \) = \pm \s_r h^\E \( \cA_{\Phi^\pm} \Psi_2, \Psi_1 \).
\end{equation}
A metric compatible real or a quaternionic structure on $E$ can be regarded as an element of $\cF^+$ or $\cF^-$, respectively.

\bigskip

\section{The Bogoliubov--de Gennes equation}
\label{app:BdG}

In this section we show how the above theory also leads to a geometric framework (and generalization) to the Bogoliubov--de Gennes equations; cf \cite{CJNPS10}*{Equation~(55)}.

For simplicity, let $X$ be a closed, oriented, Riemannian manifold. The generalized Jackiw--Rossi operator in \cref{eq:JR_op} has two uncommon properties: it is not complex linear and it may not be (real) self-adjoint. There is however a canonical way to construct a (in fact, two) complex linear and (complex) self-adjoint operator(s) out of this data.

Let $\p \in \{ - 1, 1 \}$. Using the notation of \Cref{definition:JR_op}, let $\widehat{\E} \eqdef \E \oplus \overline{\E}$. Since $\cA_\Phi$ is a complex \emph{linear} map from $\E$ to $\overline{\E}$, we have that the operator
\begin{equation}
	\widehat{H}_{\nabla, \Phi, \p} \eqdef \begin{pmatrix} \D_\nabla & \cA_\Phi \\ \cA_\Phi^* & \p \D_\nabla \end{pmatrix} = \begin{pmatrix} \D_\nabla & \cA_\Phi \\ \s_\Phi \cA_\Phi & \p \D_\nabla \end{pmatrix},
\end{equation}
is a complex linear map from $L_1^2 \( \widehat{\E} \)$ to $L^2 \( \widehat{\E} \)$ and self-adjoint as a densely defined operator on $L^2 \( \widehat{\E} \)$. Motivated by \cite{CJNPS10}*{Equations~(9)~and~(11)}, we call $\widehat{H}_{\nabla, \Phi}$ the \emph{Bogoliubov--de Gennes operator}.

For each $m \in \rl$, the equation
\begin{equation}
	\widehat{H}_{\nabla, \Phi, \p} \begin{pmatrix} \Psi_1 \\ \overline{\Psi}_2 \end{pmatrix} = m \begin{pmatrix} \Psi_1 \\ \overline{\Psi}_2 \end{pmatrix},
\end{equation}
is the Euler--Lagrange equation of the energy density
\begin{equation}
	E_{\nabla, \Phi} \( \Psi_1, \Psi_2 \) = h \( \Psi_1, \D_\nabla \Psi_1 \) + \p h \( \Psi_2, \D_\nabla \Psi_2 \) + 2 \Re \( \Phi \( \cB \( \Psi_1 \otimes \Psi_2 \) \) \) - m \( |\Psi_1|^2 + |\Psi_2|^2 \).
\end{equation}

Let us define a unitary operator on $\widehat{\E}$ as
\begin{equation}
	\widehat{C} \begin{pmatrix} \Psi_1 \\ \overline{\Psi}_2 \end{pmatrix} \eqdef \begin{pmatrix} \Psi_2 \\ \s_\Phi \p \overline{\Psi}_1 \end{pmatrix}.
\end{equation}
Then we have
\begin{align}
	\widehat{H}_{\nabla, \Phi, \p} \circ \widehat{C}	&= \p \widehat{C} \circ \widehat{H}_{\nabla, \Phi, \p}, \label{eq:HC_commutation} \\
	\widehat{C}^2										&= \s_\Phi \p \id_{\widehat{\E}}. \label{eq:C_square}
\end{align}
By \cref{eq:HC_commutation}, $\widehat{C}$ preserves the kernel of $\widehat{H}_{\nabla, \Phi, \p}$ and by \cref{eq:C_square}, $\widehat{C}$ is either a real $(\s_\Phi \p = 1)$ or a quaternionic $(\s_\Phi \p = - 1)$ structure on $\ker \( \widehat{H}_{\nabla, \Phi, \p} \)$.

If $\p = 1$, then \cref{eq:HC_commutation} can be interpreted as $\widehat{H}_{\nabla, \Phi, \p}$ having a \emph{time-reversal} symmetry, which is represented by $\widehat{C}$, and, by \cref{eq:C_square}, this symmetry is bosonic if $\s_\Phi = 1$ and fermionic if $\s_\Phi = - 1$.

If $\p = - 1$, then \cref{eq:HC_commutation} can be interpreted as $\widehat{H}_{\nabla, \Phi, \p}$ having a \emph{particle-hole} symmetry, which is represented by $\widehat{C}$, and, by \cref{eq:C_square}, this symmetry is bosonic if $\s_\Phi = - 1$ and fermionic if $\s_\Phi = 1$.

	%========================
	\bibliography{references}
	%========================

\end{document}